\documentclass[11 pt]{amsart}
\usepackage{graphicx,amsfonts,amssymb,amsmath,amsthm,url,
  verbatim, fullpage}
\usepackage[normalem]{ulem}

\newcommand{\leftexp}[2]{{\vphantom{#2}}^{#1}%
      \kern-\scriptspace%
      {#2}}
\newcommand{\Aut}{{\rm Aut}}
\newcommand{\Tr}{\mathrm{Tr}}
\newcommand{\CM}{\mathrm{CM}}
\newcommand{\tr}{\mathrm{tr}}

\renewcommand{\Im}{\mathrm{Im}}

\newcommand{\Sp}{\mathrm{Sp}}
\renewcommand{\H}{\mathbb H}
\newcommand{\T}[1]{{#1}^t}

\newcommand{\A}{{\mathbb A}}

\newcommand{\Q}{{\mathbb Q}}
\newcommand{\Z}{{\mathbb Z}}
\newcommand{\R}{{\mathbb R}}

\newcommand{\C}{{\mathbb C}}
\newcommand{\vol}{{\mathrm {vol}}}
\newcommand{\bs}{\backslash}

\newcommand{\GL}{{\rm GL}}

\newcommand{\SL}{{\rm SL}}

\newcommand{\GSp}{{\rm GSp}}

\newcommand{\sgn}{{\rm sgn}}

\newcommand{\cond}{{\rm cond}}

\newcommand{\mat}[4]
{{\setlength{\arraycolsep}{0.5mm}\left
(\begin{array}{cc}#1&#2\\#3&#4\end{array}\right)}}

\newcommand{\forget}[1]{}

\newtheorem{lemma}{Lemma}[section]
\newtheorem{theorem}{Theorem}

\newtheorem{corollary}[lemma]{Corollary}
\newtheorem{proposition}[lemma]{Proposition}
\newtheorem{definition}[lemma]{Definition}

\theoremstyle{remark}
\newtheorem{remark}[lemma]{Remark}

\begin{document}

\bibliographystyle{plain}

\title{On ratios of Petersson norms for Yoshida lifts}

\author{Abhishek Saha}
\address{ Departments of Mathematics \\
  University of Bristol \\
  Bristol, United Kingdom} \email{abhishek.saha@bristol.ac.uk}

\begin{abstract}We explicate various features of the correspondence between scalar valued Siegel cusp forms of degree $n$ and certain automorphic representations on $\GSp_{2n}$. We prove an algebraicity property for a ratio of Petersson norms associated to a Siegel cusp form of degree 2 (and arbitrary level) whose adelization generates a weak endoscopic lift.
\end{abstract}

 \maketitle
\renewcommand{\theenumi}{\roman{enumi}}

\section{Introduction}

\subsection{Motivation}It is of fundamental interest to understand the \emph{arithmetic} properties of automorphic forms on Shimura varieties. A well-known problem in this area is Deligne's conjecture~\cite{deligneconj} on algebraicity of  critical values of motivic $L$-functions. Despite results in many special cases, 
the general form of this conjecture remains far from settled. A closely related problem is to prove the algebraicity of ratios of Petersson norms (more generally, periods) for ``functorially related" automorphic forms on different Shimura varieties. 

One of the first results on the algebraicity of such period ratios was obtained for automorphic forms on indefinite quaternion algebras. Let $g$ be an arithmetically normalized holomorphic newform of even weight on a compact quotient $X$ of an indefinite quaternion algebra $D / \Q$. Let $f$ be the arithmetically normalized holomorphic newform for $\Gamma_0(N)$ associated to $g$ by the Jacquet--Langlands--Shimizu correspondence.  Let $\langle f, f \rangle$ and $\langle g, g \rangle$ denote the Petersson inner products taken on  $\Gamma_0(N) \bs \H$ and $X$ respectively. Then,  it was proved by Shimura~\cite{shimura81} that the ratio $\langle f, f \rangle / \langle g, g \rangle$ is algebraic. This result was refined by Harris--Kudla~\cite{harkud} and Prasanna~\cite{prasanna06}.

Moving on to higher rank groups, let $k$ be an even integer and $S_k(\Sp_{2n}(\Z))$ be the space of holomorphic Siegel cusp forms of weight $k$ for the group $\Sp_{2n}(\Z)$. Given any normalized Hecke eigenform $f \in S_{2k -2}(\SL_2(\Z))$, it is known (see~\cite{eichzag}) that there exists a (unique up to multiples) Hecke eigenform $F \in S_{k }(\Sp_4(\Z))$ satisfying \begin{equation}\label{SKeq}L_\mathfrak{f}(s, \pi_F) = \zeta(s - 1/2) \zeta(s + 1/2)  L_\mathfrak{f}(s, \pi_f)\end{equation}
where $\pi_f$ (resp. $\pi_F$) denotes the cuspidal automorphic representation on $\GL_2(\A)$ (resp. $\GSp_4(\A)$) associated to\footnote{See Section~\ref{s:adelizationrep} for details on how to associate automorphic representations to Siegel cusp forms.} $f$ (resp. $F$), $\zeta(s)$ denotes the  Riemann zeta function, $L_\mathfrak{f}(s, \pi_f)$ denotes the finite part of the Langlands $L$-function attached to $\pi_f$, and $L_\mathfrak{f}(s, \pi_F)$ denotes the finite part of the spin (degree 4) $L$-function attached to $\pi_F$. The Siegel cusp form $F$ is called the Saito--Kurokawa lift of $f$. By multiplying $F$ by a suitable complex number if necessary, we can assume that all its Fourier coefficients are totally real and algebraic.

Given $f$, $F$ as above and $\sigma \in \Aut(\C)$, we have forms ${}^\sigma\!f \in S_{2k -2}(\SL_2(\Z))$, ${}^\sigma\!F \in S_{k }(\Sp_4(\Z))$ via the action of $\sigma$ on Fourier coefficients. It can be shown that ${}^\sigma\!F$ is the Saito-Kurokawa lift of ${}^\sigma\!f$. Furusawa~\cite{fur84} proved the remarkable result  that the quantity $\frac{\pi^2\langle F , F \rangle}{\langle f, f \rangle}$ is an algebraic number and moreover, for all $\sigma \in \Aut(\C)$, we have $$\sigma \left(\frac{\pi^2 \langle F , F \rangle}{\langle f, f \rangle}\right)= \frac{ \pi^2 \langle {}^\sigma\!F , {}^\sigma\!F \rangle}{\langle {}^\sigma\!f, {}^\sigma\!f \rangle}.$$

Furusawa's result was refined by Kohnen~\cite{kohsk}, extended to the case of Hilbert-Siegel modular forms by Lanphier~\cite{lanphier04} and generalized to the case of Saito-Kurokawa lifts with levels by Brown~\cite{brown07} (see also Brown--Pitale~\cite[Corollary 4.2]{brown-pitale}). In a different direction, Ikeda~\cite{ikeda01} proved a lifting from $S_{2k -2n}(\SL_2(\Z))$ to  $S_{k }(\Sp_{4n}(\Z))$ whenever $k \equiv n \pmod{2}$. The Ikeda lift generalizes the Saito-Kurokawa lift. Choie--Kohnen~\cite{choiekohn03} extended Furusawa's result on algebraicity of ratios of Petersson norms to the setup of the Ikeda lift.  If $f \in S_{2k -2n}(\SL_2(\Z))$ is a normalized Hecke eigenform and $F \in S_{k }(\Sp_{4n}(\Z))$ is its Ikeda lift, then the result of  Choie--Kohnen implies that the quantity $\frac{\pi^{(n+1)(2n-1)}\langle F , F \rangle}{\langle f, f \rangle^n}$ is an algebraic number.

Other results on algebraicity of period ratios include Shimura's work~\cite{shimura81half} on forms related by the Shimura correspondence and Furusawa's result~\cite{fur84} on the holomorphic base change lift for $\GL(2)$. The results of Raghuram and Shahidi~\cite{raghushah} on period relations for cusp forms on $\GL_n$ also lie in this category. Yet another example --- perhaps more fundamental than any of the ones described so far --- is provided by the algebraicity of the ratio of Petersson norms of two Siegel cusp forms that are Hecke eigenforms  with equal eigenvalues at all but finitely many primes. This result is due to Garrett~\cite{gar2} when the eigenvalues generate a totally real field while the more general result is proved in this paper (Corollary~\ref{corpeteq}).
\subsection{A brief overview of Yoshida lifts}The results quoted in the previous subsection are all of the following type: given an (arithmetic) automorphic form on some group $G$ that is constructed by lifting form(s) on a (possibly different) group, the ratio of some appropriate powers of their Petersson norms is algebraic (up to multiplication by a power of $\pi$) and behaves nicely under the action of $\Aut(\C)$. For $G= \GSp_4$, the prototypical example of this phenomenon is provided by the  Saito-Kurokawa lifts. As described above, the algebraicity result in this case is due to Furusawa with later generalizations by various other people.

By~\eqref{SKeq}, the $L$-function of a Saito-Kurokawa lift has the property that it factors into a product of simpler $L$-functions. There is another class of arithmetically significant forms on $\GSp_4$ with the same property. Such forms come from \emph{endoscopic} representations and are classically known as Yoshida lifts. The main result of this paper is an algebraicity relation involving Petersson norms for these lifts. Formally, our relation turns out to be identical to the one for Saito-Kurokawa lifts.

Let us explain in more detail what a Yoshida lift is. We warn the reader that the definition of Yoshida lifts below may appear different than in some other works (such as~\cite{bocsch1991} or~\cite{sahaschmidt}); however, we will show in Section~\ref{s:yoshida} that our definition is equivalent to the classical definition whenever the latter is applicable. Let $ f \in S_{k_1}(N_1, \chi_1)$, $g \in S_{k_2}(N_2, \chi_2)$ be classical holomorphic newforms where the weights $k_1$, $k_2$ are both at least 2. We assume that the characters $\chi_1$, $\chi_2$ are associated with the same primitive character\footnote{In the literature, this is sometimes referred to as $\chi_1$ and $\chi_2$ being co-trained.}, and that $f$ and $g$ are not multiples of each other. Let $\pi_f$, $\pi_g$ be the irreducible cuspidal representations of $\GL_2(\A)$ attached to $f$, $g$ respectively (see~\cite{gelbook}); note that $\pi_f$, $\pi_g$ have the same central character. Let $F $ be a non-zero scalar valued holomorphic Siegel cusp form of degree 2 with respect to some congruence subgroup of $\Sp_4(\Q)$ (see Section~\ref{s:siegeldefs}). We say that $F$ is a Yoshida lift associated to the pair $(f, g)$ if the following conditions are met:

\begin{enumerate}
\item \label{cond1}The adelization\footnote{See Section~\ref{s:adelizationrep} for further details on this.} of $F$ generates an irreducible cuspidal automorphic representation $\pi_F$ of $\GSp_4(\A)$.

 \item \label{cond2}The local $L$-parameter\footnote{See Section~\ref{s:yoshidarep} for further details on this.} for $\pi_{F,v}$ at each place $v$ is the direct sum of the $L$-parameters for $\pi_{f,v}$ and $\pi_{g,v}$.

\end{enumerate}

We remark that Condition~\eqref{cond1} automatically implies that $F$ is an eigenfunction of all the local Hecke algebras outside a finite set of primes (see Proposition~\ref{propneareq}). We  also remark that Condition~\eqref{cond2} says that the Yoshida lift is a special case of Langlands functoriality, coming from the embedding of dual groups
\begin{align}\label{dualgroupmorphismeq}
 \{(g_1,g_2)\in\GL_2(\C)\times\GL_2(\C)\:|\:\det(g_1)=\det(g_2)\}&\longrightarrow\GSp_4(\C),\\
 (\mat{a}{b}{c}{d},\mat{a'}{b'}{c'}{d'})&\longrightarrow\left(\begin{matrix}a&&b\\&a'&&b'\\c&&d\\&c'&&d'\end{matrix}\right).\nonumber
\end{align}
Furthermore Condition~\eqref{cond2} implies that for \emph{any} irreducible cuspidal representation $\pi$ of $\GL_n(\A)$, we have an equality of global standard times spin $L$-functions
\begin{equation}L(s,  \pi\times\pi_{F} ) = L(s, \pi \times \pi_{f})L(s, \pi \times \pi_{g}).
\end{equation}

Note also that our assumption that $\pi_f$, $\pi_g$ have the same central character is essential, for otherwise the direct sum of their $L$-parameters cannot be an $L$-parameter for $\GSp_4$.

It is worth asking when Yoshida lifts exist. Thankfully, a complete answer is available: $f$, $g$ as above have a Yoshida lift (i.e., a scalar valued holomorphic Siegel cusp form of degree 2 satisfying the two conditions above) \emph{if and only if both the following conditions are met}:

\begin{enumerate}

\item The integers $k_1$ and $k_2$ are even, and at least one of them is equal to 2.

\item There is a prime $p$ such that $\pi_f$, $\pi_g$ are both discrete series at $p$.\footnote{Any such prime must divide $\gcd(N_1, N_2)$; in particular $N_1$ and $N_2$ cannot be coprime.}

\end{enumerate}

Furthermore, one can show that whenever a Yoshida lift exists, it must have weight $\frac{k_1}2 + 1$ (if we assume $k_2 =2$). We refer the reader to Section~\ref{s:yoshidarep} for further details on all these facts.

\begin{remark}Incidentally, the isomorphism class of $\pi_F$ is not uniquely determined by $f, g$. In fact, if $T \neq \emptyset$ is the set of finite places where $\pi_f$, $\pi_g$ are both discrete series, then the various possible Yoshida lifts of $(f, g)$ in the above sense give rise to $2^{\#T-1}$ inequivalent cuspidal representations of $\GSp_4(\A)$ all of which lie in the same global $L$-packet; see Theorem~\ref{t:yoshida}.
\end{remark}

The theory of such lifts was initiated by Yoshida in \cite{yosh1980}. Classical Yoshida lifts have been studied extensively by B{\"o}cherer and Schulze-Pillot~\cite{bocsch, bocsch1991, bocsch1994, bocsch1997}. We also refer the reader to~\cite{sahaschmidt} for a representation theoretic treatment of the classical Yoshida lifting. In all these works, the Yoshida lifts were constructed with respect to a Siegel congruence subgroup of square-free level for $\Sp_4(\Z)$.   Our definition of Yoshida lifts in this paper allows for the lift to be defined with respect to an arbitrary congruence subgroup of $\Sp_4(\Z)$.\footnote{We will prove in Proposition~\ref{equivclassical} that the definition of Yoshida lift in this paper coincides with the classical definition whenever $f$, $g$ have squarefree level and trivial central character and the lift is defined with respect to a Siegel-type congruence subgroup of $\Sp_4(\Z)$ of square-free level.} An advantage of our present approach is that we do not need to place the additional hypothesis on $f$, $g$ (namely that their levels be squarefree, their characters be trivial, and their Atkin-Lehner eigenvalues coincide at all places) that are necessary if one demands that the Yoshida lift be with respect to a Siegel congruence subgroup of $\Sp_4(\Z)$ of square-free level. In Section~\ref{s:yoshida}, we prove various technical facts about Yoshida lifts and the related notions of ``weak Yoshida lifts", ``Yoshida spaces", and ``weak Yoshida spaces".

It seems appropriate to point out here that on the \emph{automorphic representations side}, an essentially complete account of Yoshida liftings was given by Brooks Roberts~\cite{rob2001} using global theta lifts; see also the paper of Ichino-Ikeda~\cite[Sec. 8]{ichino-ikeda}. Our thrust here is slightly different. What we do in Section~\ref{s:yoshidarep} is to use the results of Roberts, along with the theory of ``adelization" and ``de-adelization" (developed in detail in Section~\ref{s:adelization}) to give an equally complete account of Yoshida liftings on the \emph{Siegel cusp forms side}.  The fundamental result  which makes this passage  possible  is Theorem~\ref{t:deadelize}, which asserts that any cuspidal representation of $\GSp_4(\A)$ with the correct archimedean component contains a vector in the image of our adelization map.

\subsection{The main result}
Let $f \in S_{2k}(N_1, \chi_1)$, $g \in S_{2}(N_2, \chi_2)$ be classical holomorphic newforms which satisfy the conditions for a scalar valued Yoshida lifting. In other words: $f$ is not a multiple of $g$, the characters $\chi_1$ and $\chi_2$ are associated to the same primitive character, and there is a prime $p$ such that $\pi_f$, $\pi_g$ are both discrete series at $p$. Let $F \in S_{k+1}(\Gamma)$ be a Yoshida lift of $(f,g)$ where $\Gamma$ is a congruence subgroup of $\Sp_4(\Z)$. Define Petersson inner products as follows:

$$\langle g, g\rangle = \frac{1}{\vol(\Gamma_0(N_2) \bs \H)}\int_{\Gamma_0(N_2) \bs \H} |g(z)|^2  dx dy,$$

$$\langle f, f\rangle =  \frac{1}{\vol(\Gamma_0(N_1) \bs \H)} \int_{\Gamma_0(N_1) \bs \H}|f(z)|^2 y^{2k - 2} dx dy,$$

$$\langle F, F\rangle = \frac{1}{\vol(\Gamma \bs \H_2)}\int_{\Gamma \bs \H_2} |F(Z)|^2 (\det Y)^{k - 2} dX dY.$$

Above, we have normalized the Petersson inner products so that they are \emph{independent} of the choice of congruence subgroups for the various forms. We now state a variant of our main result.
\medskip

\noindent\textbf{Theorem~\ref{t:main}A}. \ \emph{Let $f \in S_{2k}(N_1, \chi_1)$, $g \in S_{2}(N_2, \chi_2)$ be newforms that satisfy the conditions for a scalar valued Yoshida lifting. Let $F$ be a Yoshida lift of $(f,g)$. Suppose, further, that $k \ge 6$ and that all the Fourier coefficients of $f, g$ and $F$ are algebraic numbers contained in some CM field. Then $$\frac{\pi^2\langle F, F\rangle}{\langle f, f\rangle} \in \overline{\Q},$$ and moreover, for $\sigma \in \Aut(\C)$, we have $$\sigma\left(\frac{\pi^2\langle F, F\rangle}{\langle f, f\rangle}\right) =  \frac{\pi^2\langle {}^\sigma\!F, {}^\sigma\!F\rangle}{\langle {}^\sigma\!f, {}^\sigma\!f\rangle}.$$}
\medskip

The stipulation that $f$ and $g$ have CM algebraic Fourier coefficients is automatically met whenever their first Fourier coefficients are normalized to be equal to 1. The condition that $F$ has CM algebraic Fourier coefficients does not lead to significant loss of generality either. This is because, given any $f,g$ as above satisfying the conditions for a scalar valued Yoshida lifting, one can always find a Yoshida lift of $(f,g)$ with CM algebraic Fourier coefficients. In fact, much more is true: given any Yoshida lift $F$ of $(f,g)$, one can always find another Yoshida lift $F'$ of $(f,g)$ such that $F$ and $F'$ generate the \emph{same} cuspidal automorphic representation and the Fourier coefficients of $F'$ are contained in the $\CM$ field generated by the Hecke eigenvalues of $f$ and $g$. See Theorem~\ref{rationalyoshida} for further details.

We note here that the main result of this paper (Theorem~\ref{t:main}) is quite a bit stronger than the variant (Theorem~\ref{t:main}A) that was stated above. Indeed, Theorem~\ref{t:main} applies to any $F$ that lies in the \emph{weak Yoshida space} attached to $f, g$. The weak Yoshida space is the vector space consisting of all scalar valued Siegel cusp forms whose adelizations lie in the weak endoscopic space attached to $f$ and $g$ (see Section~\ref{s:weakyosh} for more details). All Yoshida lifts as defined above are members of the weak Yoshida space; moreover, if one assumes a special case of Langlands functoriality, then one can show that the weak Yoshida space is precisely the space spanned by all the Yoshida lifts (see Section~\ref{s:weakyosh}).

Let us say a few words about the proof of Theorem~\ref{t:main} (and it's weaker variant Theorem~\ref{t:main}A). The proof proceeds via the path of algebraicity of special $L$-values of a certain \emph{degree 8} $L$-function. The main idea is to introduce an auxiliary classical newform $h$ and consider the $\GSp_4 \times \GL_2$ $L$-function for $\pi_F \times \pi_{h}$. The theorem is deduced by combining special value results for this $L$-function due to Morimoto~\cite{morimoto} with special value results for the Rankin--Selberg $L$-function on $\GL_2 \times \GL_2$ due to Shimura~\cite{shi76}. We remark that our proof gives further evidence of the connection between period relations on functorially related automorphic forms and special value results in the spirit of Deligne's conjecture. We also remark that in the special case of classical Yoshida lifts of \emph{squarefree level} (with respect to a Siegel-type congruence subgroup), Theorem~\ref{t:main}A can also be derived from an inner product formula proved by B\"ocherer and Schulze-Pillot~\cite{bocsch1991} and generalized by  B\"ocherer--Dummigan--Schulze-Pillot~\cite{boschdum}.

\subsection{Further remarks}There is a well-known
 procedure~\cite{gelbook} for associating an automorphic representation of $\GL_2(\A)$ to a classical (elliptic) modular form. In Section~\ref{s:adelization}, which has a semi-expository flavour, we explicitly describe an analogous procedure for attaching automorphic representations of $\GSp_{2n}(\A)$ to (scalar-valued) Siegel cusp forms for arbitrary congruence subgroups of $\Sp_{2n}(\Z)$.\footnote{In the case of elliptic modular forms, there is no loss of generality in restricting oneself to congruence subgroups of the form $\Gamma_1(N)$ but for Siegel modular forms of higher degree it is really necessary to consider various types of congruence subgroups.} We also show, using a recent result of Roberts and Schmidt \cite{roberts-schmidt13}, that every cuspidal automorphic representation of $\GSp_4(\A)$ with the right component at infinity can be obtained from a suitable Siegel cusp form of degree 2 (see Theorem~\ref{t:deadelize} for the precise statement). These ideas are used in Section~\ref{s:arithmeticity} to prove various arithmeticity properties for Siegel cusp forms. They are also the foundation for Section~\ref{s:yoshida}, where we give a representation-theoretic account of Yoshida liftings of arbitrary level and prove various properties of these lifts. In the interest of exposition, and keeping in mind certain future projects, we have developed the theory in Sections~\ref{s:adelization} and \ref{s:yoshida} further than needed for the main result of this paper. In fact, several of the basic results stated therein  do not seem to exist in an equivalent formulation elsewhere. We hope that these parts of the paper will be of independent interest to researchers working in the area.

We have restricted to scalar valued Yoshida lifts in this paper. However, a fuller treatment should include the case of vector valued forms. One reason we restrict to the scalar valued case is for simplicity; the other is that Morimoto's special value result for $\GSp_4 \times \GL_2$ (quoted here as Theorem~\ref{tgsp4gl2}), which is a key ingredient for the proof of our main result (Theorem~\ref{t:main}), is currently available only for scalar valued Siegel cusp forms. If his result can be generalized to include the case of vector valued forms, then our main result will extend to the case of vector valued Yoshida lifts (which occur when neither of the classical newforms $f$, $g$ being lifted has weight equal to 2).

Finally, it seems apt to mention a certain lifting from classical newforms that produces arithmetically significant Siegel cusp forms of degree 2 which are neither of Saito-Kurokawa type, nor of Yoshida type. This is the symmetric cube lift constructed by Ramakrishnan and Shahidi~\cite{ramshi}. If $f$ is a classical newform and $F$ is its symmetric cube lift, then there is good reason to believe that there exists an algebraicity result for $\frac{\langle F, F\rangle}{\langle f, f\rangle^3}$ (see~\cite[Conj. 3.3]{ibukat13}). The methods of this paper do not apply easily in that situation and a more ``motivic" approach seems to be necessary; this is the subject of current investigation in collaboration with A. Raghuram.

\subsection{Acknowledgements}I would like to thank Ralf Schmidt for some illuminating discussions about adelization. I would also like to thank Paul Nelson, Ameya Pitale, and A. Raghuram for helpful comments on an earlier version of this manuscript.

\section{Notations and some basic results for Siegel cusp forms}

\subsection{Notations}The symbols $\Z$, $\Q$, $\R$, $\C$, $\Z_p$ and $\Q_p$ have the usual meanings. We let $\A = \otimes_v' \Q_v$ denote the ring of adeles over $\Q$ and $\A_\mathfrak{f} = \otimes_p' \Q_p$ the subring of finite adeles. The phrase ``almost all" will mean ``all except for finitely many". Given a Dirichlet character $\chi$, we let $\cond(\chi)$ denote the conductor of $\chi$ and we define the Gauss sum $$G(\chi) = G(\chi_0) = \sum_{n=1}^{\cond(\chi)} \chi_0(n) e^{2\pi i n /N},$$ where $\chi_0$ is the primitive Dirichlet character associated to $\chi$. Given a family of subfields $\{F_x\}$ of $\C$, we let $\coprod_x F_x$ denote their compositum in $\C$. We let $\Q^{\tr}$ denote the compositum of all totally real number fields and we denote $\Q^{\CM} =  \Q^{\tr}(i)$. It is well-known (see, for instance~\cite{washcycl}) that $\Q^{\CM}$ is the compositum of all CM fields and that any number field that is a subfield of $\Q^{\CM}$ is either a CM field or a totally real number field.

For any commutative ring $R$ and positive integer $n$, let $M_n(R)$
  denote the ring of $n$ by $n$ matrices with entries in $R$ and $\GL_n(R)$ denote the group of invertible matrices.  If
  $A\in M_n(R)$, we let $\T{A}$ denote its transpose. We let $M_n^{\rm sym}(R)$ denote the additive group of symmetric matrices in  $M_n(R)$. Let $I_n$ denote the $n$ by $n$ identity matrix. Denote by $J_n$ the $2n$ by $2n$ matrix given by
$$
J_n =
\begin{pmatrix}
0 & I_n\\
-I_n & 0\\
\end{pmatrix}.
$$

Define the algebraic groups
   $\GSp_{2n}$ and $\Sp_{2n}$ over $\Z$ by
$$\GSp_{2n}(R) = \{g \in \GL_{2n}(R) \; | \; \T{g}J_ng =
  \mu_n(g)J_n,\:\mu_n(g)\in R^{\times}\},$$
$$
\Sp_{2n}(R) = \{g \in \GSp_{2n}(R) \; | \; \mu_n(g)=1\},
$$
for any commutative ring $R$. \emph{Throughout this paper, we let the letter $G_n$ denote the group $\GSp_{2n}$.} If $S$ is any subring of $\R$, we let $G_n(S)^+$ denote the subgroup of $G_n(S)$ consisting of elements with $\mu_n(g)>0$. We let $K_\infty$ denote the standard maximal compact subgroup of $G_n(\R)^+$ given by $$K_\infty = \left\{ \mat{A}{B}{-B}{A}  \in \GL_{2n}(\R), \ A \T{A} + B \T{B} = I_n, A\T{B} = B \T{A} \right\}.$$

The Siegel upper-half space of degree $n$ is defined by
$$
\H_n = \{ Z \in M_n^{\rm sym}(\C)\;|\ \Im(Z)
  \text{ is positive definite}\}.
$$ We will often write $Z = X + iY$ where $X$ and $Y$ lie in $M_n^{\rm sym}(\R)$.
The group $G_{n}(\R)^+$ acts on $\H_n$ via
$$
 gZ := (AZ+B)(CZ+D)^{-1}\qquad\text{for }
 g=\begin{pmatrix} A&B\\ C&D \end{pmatrix} \in G_{n}(\R)^+,\;Z\in \H_n.
$$
We let $J(g,Z) = CZ + D$.
For any positive integer $N$,
define
\begin{equation}\label{Gamma0defeq}
\Gamma^{(n)}(N) := \left\{g\in \Sp_{2n}(\Z)\;|\;g \equiv I_{2n} \pmod{N}\right\}
\end{equation}
 and
 \begin{equation}\label{Gamma0sdefeq}
\Gamma^{(n), *}(N) := \left\{g\in \Sp_{2n}(\Z) \;|\;g \equiv \mat{aI_n}{}{}{a^{-1}I_n} \pmod{N}, \text{ for some } a \in (\Z/N\Z)^\times.\right\}
\end{equation}

We say that a subgroup $\Gamma$ of $\Sp_{2n}(\Q)$ is a \emph{congruence subgroup} if there is an integer $N$ such that $\Gamma$ contains $\Gamma^{(n)}(N)$ as a subgroup  of finite index.

\subsection{Siegel cusp forms}\label{s:siegeldefs}Given a congruence subgroup $\Gamma$ of $\Sp_{2n}(\Q)$ and a positive integer $k$, let $S_k(\Gamma)$ denote the space of holomorphic functions $F$ on
$\H_n$ which satisfy the relation
\begin{equation}\label{siegeldefiningrel}
F(\gamma Z) = \det(J(\gamma,Z))^k F(Z), \quad \gamma \in \Gamma, Z \in \H_n,
\end{equation}
and vanish at all the
cusps.\footnote{For a precise formulation of this cusp vanishing condition, see~\cite[I.4.6]{Fr1991}.} It is well-known that $S_k(\Gamma)$ is a finite-dimensional vector space for each $k$, $\Gamma$. The integer $n$ is called the degree of the Siegel cusp form $F$. We denote $$S_k^{(n)} = \bigcup_{\Gamma} S_k(\Gamma) $$ where $\Gamma$ varies over all congruence subgroups of $\Sp_{2n}(\Q)$. Note that $S_k^{(n)} = \bigcup_{N \geq 1} S_k(\Gamma^{(n)}(N))$.

Given any primitive Dirichlet character $\chi$ and any integer $N$ such that $\cond(\chi) | N$, let $S_k(\Gamma^{(n),*}(N),$ $\chi)$ denote the subspace of  $S_k(\Gamma^{(n)}(N))$ consisting of those $F$ which satisfy
\begin{equation}\label{siegeldefiningrel2}
F(\gamma Z) = \chi^{-1}(a) \det(J(\gamma,Z))^k F(Z)
\end{equation} for $\gamma \in \Gamma^{(n),*}(N), \quad \gamma \equiv \mat{aI_n}{}{}{a^{-1}I_n} \pmod{N}.$ For any primitive Dirichlet character $\chi$, we denote  \begin{equation}\label{siegeldefiningrel3}S_k^{(n)}(\chi) = \bigcup_{N} S_k(\Gamma^{(n), *}(N), \chi)\end{equation} where the union is taken over all $N$ that are multiples of $\cond(\chi)$.

\begin{lemma}\label{lemmachardec}\begin{enumerate}
\item For each $N$, we have $S_k(\Gamma^{(n)}(N)) = \oplus_{\chi} S_k(\Gamma^{(n),*}(N), \chi)$ where the direct sum is taken over all primitive Dirichlet characters whose conductor divides $N$.
\item We have $S_k^{(n)}= \oplus_{\chi} S_k^{(n)}(\chi)$ where the direct sum is taken over all primitive Dirichlet characters.
 \item  $S_k^{(n)}(\chi) \neq \{0\}$ implies that $\chi(-1) = (-1)^{nk}$.
\end{enumerate}
\end{lemma}
\begin{proof}The first assertion follows from Fourier analysis on the finite abelian group $\Gamma^{(n),*}(N) / \Gamma^{(n)}(N)$ $\cong (\Z/N\Z)^\times$. The second assertion follows immediately from the first using~\eqref{siegeldefiningrel3}. The third assertion follows by putting $\gamma = - I_{2n}$ in~\eqref{siegeldefiningrel2}.

\end{proof}

Now, let $F \in S_k^{(n)}$. We have the Fourier expansion
$$F(Z) = \sum_{T\in M_n^{\rm sym}(\Q)}a(F, T) e^{2 \pi i \Tr TZ}$$ where the Fourier coefficients $a(F, T)$ are zero unless $T$ is inside a lattice of the form $\frac1NM_n^{\rm sym}(\Z)$.
For any subfield $L$ of $\C$ and any congruence subgroup $\Gamma$ of $\Sp_{2n}(\Q)$, we define the $L$-vector space
\begin{equation}S_k(\Gamma;L) = \{F \in S_k(\Gamma) : a(F,T) \in L \text{ for all }T \in M_n^{\rm sym}(\Q) \}.
\end{equation}
It is a fairly deep fact~\cite[p. 67--74]{shibook2} that \begin{equation}\label{sh:rationalbasis}S_k(\Gamma^{(n)}(N)) =  S_k(\Gamma^{(n)}(N); \Q) \otimes \C.
\end{equation} It follows that given any $\sigma \in \Aut(\C)$ and $F \in S_k^{(n)}$, there exists an element
${}^\sigma\!F \in S_k^{(n)}$ whose Fourier expansion is given by \begin{equation}\label{autcact}{}^\sigma\!F(Z) = \sum_{T\in M_n^{\rm sym}(\Q)}\sigma(a(F, T)) e^{2 \pi i \Tr TZ},\end{equation} and moreover, if $F \in S_k(\Gamma^{(n)}(N))$ then ${}^\sigma\!F \in S_k(\Gamma^{(n)}(N))$.
If $L$ is an algebraic extension of $\Q$ of infinite degree and $F \in S_k(\Gamma^{(n)}(N); L)$, then  $F$ actually belongs to $S_k(\Gamma^{(n)}(N); L')$ where $L'$ is a suitable number field contained in $L$.

Given $F, G \in S_k^{(n)}$, we define their Petersson inner product as follows:
$$\langle F, G\rangle = \frac{1}{\vol(\Gamma \bs \H_n)}\int_{\Gamma \bs \H_n} F(Z)\overline{G(Z)} (\det Y)^{k - n-1} dX dY,$$ where $\Gamma$ is any congruence subgroup such that $F$ and $G$ both belong to $S_k(\Gamma)$; note that our definition is independent of this choice. One can check easily that the various subspaces $S_k^{(n)}(\chi)$, as $\chi$ varies over primitive Dirichlet characters, are pairwise orthogonal under the Petersson inner product.

Finally, we introduce another notation specifically for classical (i.e., degree 1) cusp forms. If $k$ and $N$ are positive integers and $\chi$ is a Dirichlet character mod $N$, we define $S_k(N, \chi)$ as follows:
$$S_k(N, \chi) = S_k^{(1)}(\chi') \cap S_k(\Gamma_1(N)),$$ where $\chi'$ is the primitive Dirichlet character associated to $\chi$ and $$\Gamma_1(N) := \left\{g\in \SL_{2}(\Z) \;|\;g \equiv \mat{1}{*}{0}{1} \pmod{N}\right\}.$$ This definition of $S_k(N, \chi)$ coincides with the one commonly used when talking about classical holomorphic cusp forms. An element $f$ in $S_k^{(1)}$ is called a ``classical holomorphic newform of weight $k$, level $N$, character $\chi$" if $f$ belongs to $S_k(N, \chi)$, is an eigenfunction of all the Hecke operators, and is orthogonal to all forms of the form $g(dz)$ with $g \in S_k(N', \chi)$ and $d$, $N'$ arbitary integers such that $\cond(\chi) | N' | N$, $N' \neq N$, $d| (N/N')$. We refer the reader to~\cite{atkin-lehner} and~\cite{Li79} for further details on classical newform theory.

\section{Adelization and de-adelization}\label{s:adelization}

In this section, which may be of independent interest, we explicitly describe the  procedure of attaching cuspidal automorphic representations of $G_n(\A)$ to elements of $S_k^{(n)}$. A novel feature of our treatment below is that we do not put any restriction on the level of the Siegel cusp forms, despite the lack of a good newform theory.\footnote{However, for a detailed explanation of adelization for the \emph{full level} case, including vector valued Siegel modular forms of any degree, we point the reader to the paper~\cite{asgsch} by Asgari and Schmidt, a work that has influenced the treatment here.} We also go on to prove some arithmeticity results involving Fourier coefficients and Petersson norms of Siegel cusp forms.

\subsection{The adelization map}
 Let $F$ be an element of $S_k^{(n)}$. Let $N$ be any integer such that $F\in S_k(\Gamma^{(n)}(N))$. For each prime $p$, define a compact open subgroup $K_p^N$ of $G_n(\Z_p)$ by $$K_p^N =  \left\{g\in G_n(\Z_p)\;|\;g \equiv \mat{I_n}{}{}{aI_n} \pmod{N}, \ a\in \Z_p^\times\right\}.$$ Note that our choice of $K_p^N$ satisfies the following properties:
\begin{itemize}
\item $K_p^N = G_n(\Z_p)$ for all primes $p$ not dividing $N$,
\item The multiplier map $\mu_n : K_p^N \mapsto \Z_p^\times$ is surjective for all primes $p$,
\item $
 \Gamma^{(n)}(N)=G_n(\Q) \bigcap G_n(\R)^+\prod_{p}K_p^N.
$

\end{itemize}

By strong approximation for $\Sp_{2n}$, one has
$$G_n(\A) = G_n(\Q)G_n(\R)^+\prod_{p} K_p^N.$$ Write an element $g \in G_n(\A) $ as
$$g = g_\Q g_\infty k_\mathfrak{f}, \qquad g_\Q\in G_n(\Q), g_\infty \in G_n(\R)^+, k_\mathfrak{f} \in \prod_{p}K_p^N,$$
and let the \emph{adelization} $\Phi_F$ of $F$ be the function on $G_n(\A)$ defined by $$\Phi_F(g) = \mu_n(g_\infty)^{\frac{nk}{2}} \det(J(g_\infty, iI_n))^{-k} F(g_\infty(iI_n)).$$ This is well defined because of the way the groups $K_p^N$ were chosen. Furthermore, it is independent of the choice of $N$.\footnote{It is worth noting that this construction is not completely canonical. For instance, we could have as well chosen $K_p^N$ to be the group of matrices congruent to $\mat{aI_n}{}{}{I_n} \pmod{N}$ and gotten another notion of adelization $F \mapsto \Phi_F$ which will not in general coincide with ours.} By construction, it is clear that $\Phi_F(h g) = \Phi_F(g) \ \text{ for all } h \in G_n(\Q), g \in G_n(\A)$. It is also easy to see that the mapping $F \mapsto \Phi_F$ is linear. In the next few pages, we explicate various additional properties of this mapping.

Let $Z$ denote the center of $G_n.$ For any unitary Hecke character $\chi$ of $
\Q^\times \bs \A^\times$, let $L^2(G_n(\Q)\bs G_n(\A),$ $\chi)$ denote the Hilbert space of measurable functions $\Phi$ on $G_n(\A)$ satisfying the following properties:

\begin{enumerate}
\item $\Phi(h g) = \Phi(g) \ \text{ for all } h \in G_n(\Q), g \in G_n(\A)$,\medskip

\item $\Phi(zg) =\chi(z)\Phi(g)$ for all $z \in Z(\A)$, $g \in G_n(\A)$,\medskip

\item $\int_{Z(\A)G_n(\Q)\bs G_n(\A)} |\Phi(g)|^2 dg < \infty,$\medskip

\end{enumerate}

Let $L^2_0(G_n(\Q)\bs G_n(\A), \chi)$ denote the closed subspace of $L^2(G_n(\Q)\bs G_n(\A), \chi)$ consisting of all $\Phi \in L^2(G_n(\Q)\bs G_n(\A), \chi)$ satisfying $$\int_{N(\Q) \bs N(\A)}\Phi(ng) dn = 0$$ for each $g \in G_n(\A)$, and each unipotent radical $N$ of each proper parabolic subgroup of $G_n$. The space $L^2_0(G_n(\Q)\bs G_n(\A), \chi)$ is called the \emph{cuspidal} subspace of $L^2(G_n(\Q)\bs G_n(\A),$ $\chi)$. We let $G_n(\A)$ act on $L^2(G_n(\Q)\bs G_n(\A), \chi)$ via the usual right-regular representation; this action leaves the subspace $L^2_0(G_n(\Q)\bs G_n(\A), \chi)$ invariant.

Define an inner product on $L^2(G_n(\Q)\bs G_n(\A), \chi)$ via the equation
\begin{equation}\langle \Phi_1, \Phi_2 \rangle = \frac1{\vol(Z(\A)G_n(\Q) \bs G_n(\A))}\int_{Z(\A)G_n(\Q) \bs G_n(\A)} \Phi_1(g) \overline{\Phi_2(g)} dg \end{equation} where $dg$ denotes the quotient measure induced from any Haar measure on $G_n(\A)$; our definition is independent of this choice.

Let $\mathfrak{p}^{-}_\C$ be defined as in~\cite[p.189]{asgsch}; it denotes a certain space of differential operators acting on the space of smooth functions on $G_n(\R)^+$. For any primitive Dirichlet character $\chi$,  or equivalently, any unitary Hecke character $\chi$ of $
\Q^\times \bs \A^\times$ of finite order,  let $V_k^{(n)}(\chi)$ denote the space of functions $\Phi$ on $G_n(\A)$ such that
\begin{enumerate}

\item $\Phi \in  L^2_0(G_n(\Q)\bs G_n(\A), \chi)$.
  \medskip

\item $\Phi(g k_\infty) = \det(J(k_\infty, iI_n))^{-k} \Phi(g)$ for all $k_\infty \in K_\infty$. \medskip

 \item The function $\Phi$, viewed as a function of    $G_n(\R)^+$, is smooth and satisfies the differential equation $\mathfrak{p}^{-}_\C \cdot \Phi = 0$. \medskip

 \item There exists an integer $N$ such that $\Phi(g k_\mathfrak{f}) = \Phi(g)$ for all $k_\mathfrak{f} \in \prod_{p}K_p^N$.

\end{enumerate}
It can be checked easily that $V_k^{(n)}(\chi) \neq \{0\}$ implies that $\chi_\infty(-1) = (-1)^{nk}$.

We define $V_k^{(n)}$ to be the subspace of functions $\Phi$ on $G_n(\A)$ that can be written as $\Phi = \sum_{i=1}^k  \Phi_i$ with $\Phi_i \in  V_k^{(n)}(\chi_i)$ for some $\chi_i$. Thus, we have the direct sum decomposition  $V_k^{(n)} = \oplus_\chi V_k^{(n)}(\chi)$.
Since we have equipped each space $L^2_0(G_n(\Q)\bs G_n(\A), \chi)$ with an inner product, this naturally turns the space $V_k^{(n)}$ into an inner product space where the various subspaces $V_k^{(n)}(\chi)$ are pairwise orthogonal. Note that $V_k^{(n)}$ is not closed under the right-action of $G_n(\A)$ or even under that of $G_n(\A_\mathfrak{f})$. We remark that all elements of $V_k^{(n)}$ are $(K_\infty\prod_pG_n(\Z_p))$-finite in the sense of~\cite[p. 299]{bump}.

\begin{remark}It is possible to define $V_k^{(n)}$ without any reference to the subspaces $V_k^{(n)}(\chi)$. Indeed, $V_k^{(n)}$ consists of the space of functions $\Phi$ on $G_n(\A)$ that satisfy \begin{enumerate}
\item $\Phi(h g) = \Phi(g) \ \text{ for all } h \in G_n(\Q), \  g \in G_n(\A)$,\medskip

\item $\int_{N(\Q) \bs N(\A)}\Phi(ng) dn = 0$ for each $g \in G_n(\A)$, and each unipotent radical $N$ of each proper parabolic subgroup of $G_n$.

\item $\Phi$ is invariant under the action of $Z(\R^+) \simeq \R^+$ and generates a finite dimensional space under the action of $Z(\A) \simeq \A^\times$.

\item $\int_{Z(\R^+)G_n(\Q)\bs G_n(\A)} |\Phi(g)|^2 dg < \infty,$\medskip

\item $\Phi(g k_\infty) = \det(J(k_\infty, iI_n))^{-k} \Phi(g)$ for all $k_\infty \in K_\infty$. \medskip

 \item The function $\Phi$, viewed as a function of    $G_n(\R)^+$, is smooth and satisfies the differential equation $\mathfrak{p}^{-}_\C \cdot \Phi = 0$. \medskip

 \item There exists an integer $N$ such that $\Phi(g k_\mathfrak{f}) = \Phi(g)$ for all $k_\mathfrak{f} \in \prod_{p}K_p^N$.
\end{enumerate}
\end{remark}

\medskip

\begin{definition}For any $F \in S_k^{(n)}$, and any prime $p$, we say that $F$ is $p$-spherical if there exists an integer $N$ such that $ p \nmid N$ and $F \in S_k(\Gamma^{(n)}(N))$.
\end{definition}

\begin{definition}For any $\Phi \in V_k^{(n)}$, we say that $\Phi$ is $p$-spherical if $\Phi(g k) = \Phi(g)$ for all $k \in G_n(\Z_p)$.
\end{definition}

The following facts follow directly from the definitions:
\begin{enumerate}
\item If $F\in S_k^{(n)}$, then $F$ is $p$-spherical for almost all primes $p$.

  \item  If $\Phi \in V_k^{(n)}$, then $\Phi$ is $p$-spherical for almost all primes $p$.

\item For each prime $p$, the set of $p$-spherical elements of $S_k^{(n)}$ (resp. $V_k^{(n)}$), is a subspace of $S_k^{(n)}$  (resp. $V_k^{(n)}$).
\end{enumerate}

\subsection{A quick survey of Hecke operators}
 In this subsection, we define classical and adelic local Hecke algebras at the unramified places. Let $N$ be an integer and $p$ any prime not dividing $N$. Let $\mathcal{H}_{p,N}^{\mathrm{class}, (n)}$ be the $p$-component of the classical Hecke algebra for $\Gamma^{(n)}(N)$. Precisely, it consists of $\Z$-linear combinations of double cosets $\Gamma^{(n)}(N)M\Gamma^{(n)}(N)$ with $M$ lying in the group $\Delta_{p,N}^{(n)}$ defined by $$\Delta_{p,N}^{(n)} = \{ g \in G_n(\Z[p^{-1}])^+, \  g \equiv \mat{I_n}{0}{0}{\mu_n(g)I_n} \pmod{N} \}.$$ Above,  $\Z[p^{-1}]$ denotes the subring of the rational numbers with only $p$-powers in the denominator. We define multiplication of two elements in $\mathcal{H}_{p,N}^{\mathrm{class}, (n)}$ in the usual way (see~\cite[(11.1.3)]{shibook1}), thus converting $\mathcal{H}_{p,N}^{\mathrm{class}, (n)}$ into a ring. There is a natural map $i_N: \mathcal{H}_{p,N}^{\mathrm{class}, (n)} \rightarrow \mathcal{H}_{p,1}^{\mathrm{class}, (n)}$ that is defined via $\Gamma^{(n)}(N)M\Gamma^{(n)}(N) \rightarrow \Gamma^{(n)}(1)M\Gamma^{(n)}(1)$ for each $M \in \Delta_{p,N}^{(n)}$.

\begin{lemma}\label{heckeiso} The ring $\mathcal{H}_{p,1}^{\mathrm{class}, (n)}$ is commutative and generated by $$T^{(n)}(p) = \Gamma^{(n)}(1)\mathrm{diag}(\underbrace{1, \ldots, 1}_n, \underbrace{p, \ldots, p}_n)\Gamma^{(n)}(1) $$ and the elements $$T_i^{(n)}(p^2) = \Gamma^{(n)}(1)\mathrm{diag}(\underbrace{1, \ldots, 1}_{n-i}, \underbrace{p, \ldots, p}_i, \underbrace{p^2, \ldots, p^2}_{n-i}, \underbrace{p, \ldots, p}_i)\Gamma^{(n)}(1) $$ for $i=1, \ldots, n$. Moreover, if $p$ is any prime and $N$ any integer not divisible by $p$, then the map $i_N: \mathcal{H}_{p,N}^{\mathrm{class}, (n)} \rightarrow \mathcal{H}_{p,1}^{\mathrm{class}, (n)}$  is a ring isomorphism.
\end{lemma}

\begin{proof}Both assertions are fairly well-known. For a proof of the former, see~\cite[Lemma 3.3]{andzhu}. For a proof of the latter, see for instance~\cite[Thms. 3.7 and 3.23]{andzhu}.
\end{proof}

The ring $\mathcal{H}_{p,1}^{\mathrm{class}, (n)}$ has a canonical involution induced by the map $$\Gamma^{(n)}(1)M\Gamma^{(n)}(1)\mapsto \Gamma^{(n)}(1)M^{-1}\Gamma^{(n)}(1)$$ for each $M \in \Delta_{p,N}^{(n)}$. We denote this involution by $T \mapsto T^*$.

We now define a \emph{right-action} of $\mathcal{H}_{p,1}^{\mathrm{class}, (n)}$ on the space of $p$-spherical elements of $S_k^{(n)}$. If $F \in S_k^{(n)}$ is $p$-spherical and $$T = \Gamma^{(n)}(1)M\Gamma^{(n)}(1), \quad M \in \Delta_{p,1}^{(n)},$$ then we let $N$ denote any integer such that $p \nmid N$, $F \in S_k(\Gamma^{(n)}(N))$ and define \begin{equation}\label{heckeformula}(F |_k T)(Z) =\sum_i \mu_n(M_i)^{\frac{nk}{2}} \det(J(M_i, Z))^{-k} F(M_iZ),\end{equation} where the matrices $M_i$ are given by $$i_N^{-1}\left(\Gamma^{(n)}(1)M\Gamma^{(n)}(1)\right)=\bigsqcup_i \Gamma^{(n)}(N)M_i. $$
From basic properties of the Hecke algebra (see~\cite[Thm. 3.3]{andzhu}), one can check that the mapping $F \mapsto (F |_k T)$ given by~\eqref{heckeformula} extends by linearity to a well-defined right-action of $\mathcal{H}_{p,1}^{\mathrm{class}, (n)}$ on the $p$-spherical elements of $S_k^{(n)}$. For any two $p$-spherical elements $F_1, F_2 \in S_k^{(n)}$, and any $T \in \mathcal{H}_{p,1}^{\mathrm{class}, (n)}$, one has the relation \begin{equation}\label{heckehermone}\langle F_1 |_k T, \  F_2 \rangle =  \ \langle F_1, \ F_2 |_k T^*\rangle.
\end{equation}If $p_1$ and $p_2$ are two primes, and $T_1 \in \mathcal{H}_{p_1,1}^{\mathrm{class}, (n)}$, $T_2 \in \mathcal{H}_{p_2,1}^{\mathrm{class}, (n)}$ are two Hecke operators, then for any $F \in S_k^{(n)}$ that is $p_1$-spherical as well as $p_2$-spherical, we have the commutativity relation $$(F|_k T_1)|_k T_2 = (F|_k T_2)|_k T_1.$$

Moreover, $\mathcal{H}_{p,1}^{\mathrm{class}, (n)}$ preserves the space of $p$-spherical elements in $S_k^{(n)}(\chi)$ for each primitive Dirichlet character $\chi$ of conductor co-prime to $p$. Furthermore, for two $p$-spherical elements $F_1, F_2 \in S_k^{(n)}(\chi)$, and any $M \in \Delta_{p,1}^{(n)}$, one has the relation \begin{equation}\label{heckeherm}\langle F_1 |_k (\Gamma^{(n)}(1)M\Gamma^{(n)}(1)), \  F_2 \rangle = \chi(\mu_n(M)) \ \langle F_1, \ F_2 |_k (\Gamma^{(n)}(1)M\Gamma^{(n)}(1))\rangle.
\end{equation}

\begin{proposition}\label{funprop}\begin{enumerate}
\item Let $F \in S_k(\Gamma^{(n)}(N))$ be an eigenfunction of the Hecke operators $T^{(n)}(p)$ for a subset of the primes with Dirichlet density greater than $1/2$. Then $F\in S_k(\Gamma^{(n), *}(N), \chi)$ for some $\chi$.
\item The space $S_k(\Gamma^{(n), *}(N), \chi)$ has a basis consisting of eigenforms for all the local Hecke algebras $\mathcal{H}_{p,1}^{\mathrm{class}, (n)}$ with $p \nmid N$.

\end{enumerate}
\end{proposition}
\begin{proof}Both assertions follow from the relation~\eqref{heckeherm}. For the first, we use Lemma~\ref{lemmachardec} to write $F = \sum_{i=1}^k F_i$ where $0\neq F_i \in S_k(\Gamma^{(n), *}(N), \chi_i).$ Let $\lambda_F(p)$ be the Hecke eigenvalue for the action of $T^{(n)}(p)$ on $F$. Then,~\eqref{heckeherm} implies that $\chi_i(p) = \lambda(p)/\overline{\lambda(p)}$ each $1 \le i \le k$. Hence $\chi_i(p) = \chi_j(p)$ for a subset of the primes with Dirichlet density greater than $1/2$. This, by a classical theorem of Hecke, implies that $\chi_i = \chi_j$.

For the second, we use the fact that by~\eqref{heckeherm}, the various local Hecke algebras comprise a system of commuting normal operators on $S_k(\Gamma^{(n), *}(N), \chi)$ and hence can be simultaneously diagonalized.

\end{proof}
\medskip

Next, for any prime $p$, let $\mathcal{H}_p^{(n)}$ denote the unramified  Hecke algebra of $G_n(\Q_p)$; this consists of compactly supported functions $f: G_n(\Q_p) \rightarrow \C$ which are left and right $G_n(\Z_p)$-invariant. The product in $\mathcal{H}_p^{(n)}$ is given by convolution $$f*g(x) = \frac1{\vol(G_n(\Z_p))}\int_{G_n(\Q_p)}f(xy)g(y^{-1})dy.$$

For any $M \in G_n(\Z[p^{-1}])^+$, we let $\widetilde{M} \in \mathcal{H}_p^{(n)}$ denote the characteristic function of $G_n(\Z_p)M G_n(\Z_p)$. By linearity, this gives a map $T \mapsto \widetilde{T}$ from $\mathcal{H}_{p,1}^{\mathrm{class}, (n)}$ to $\mathcal{H}_p^{(n)}$. We have the following lemma.

\begin{lemma}The map $T \mapsto \widetilde{T}$ from $\mathcal{H}_{p,1}^{\mathrm{class}, (n)}$ to $\mathcal{H}_p^{(n)}$ is a ring isomorphism.
\end{lemma}
\begin{proof}This is exactly the content of~\cite[Lemma 8]{asgsch}.
\end{proof}
For each prime $p$, we have a \emph{left action} of  $\mathcal{H}_p^{(n)}$ on the space of $p$-spherical elements of $V_k^{(n)}$. This is given by
$$(f\Phi)(g) = \frac1{\vol(G_n(\Z_p))}\int_{G_n(\Q_p)}f(h)\Phi(gh) dh, \quad f \in \mathcal{H}_p^{(n)}, \ \Phi \in V_k^{(n)},$$ where $dh$ is any Haar measure on $G_n(\Q_p)$. If $ \widetilde{M}$ is the characteristic function of $G_n(\Z_p)M G_n(\Z_p) = \sqcup_i M_i G_n(\Z_p)$ where each $M$, $M_i$ are elements of $G_n(\Q_p)$ and $\Phi$ is $p$-spherical, then one can easily verify that $$\widetilde{M}\Phi (g) = \sum_i \Phi(g M_i).$$ It is clear that $\mathcal{H}_p^{(n)}$ preserves the space of $p$-spherical elements in $V_k^{(n)}(\chi)$ for each $\chi$.

\subsection{Automorphic representations} \label{s:adelizationrep}In this subsection, we will get to the heart of the adelization construction, namely, the link between Siegel cusp forms and local and global representations of $G_n$.
\begin{theorem}\label{t:biject} The linear mapping $F \mapsto \Phi_F$ maps each subspace $S_k^{(n)}(\chi)$ isomorphically onto $V_k^{(n)}(\chi)$ and hence defines an isomorphism of vector spaces $S_k^{(n)} \simeq V_k^{(n)}$. Furthermore, this isomorphism has the following properties:

\begin{enumerate}

 \item  For each prime $p$, $F$ is $p$-spherical if and only if $\Phi_F$ is $p$-spherical.

 \item The bijection $F \rightarrow \Phi_F$ is an isometry, i.e., for all $F\in S_k^{(n)}$, we have
     $$\langle F, F \rangle = \langle \Phi_F, \Phi_F \rangle.$$

\item The bijection $F \rightarrow \Phi_F$ is Hecke-equivariant in the following sense: for all $p$-spherical $F$ in $S_k^{(n)}$ and $T \in \mathcal{H}_{p,1}^{\mathrm{class}, (n)}$, we have
     $\Phi_{F |_k T^*}  =  \widetilde{T}\Phi_F.$
\end{enumerate}
\end{theorem}
\begin{proof}Let $N$ be an integer such that  $F\in S_k(\Gamma^{(n), *}(N), \chi)$. It is easy to check that $\Phi_{F}(zg) =\chi(z)\Phi_{F}(g)$. By fairly standard methods (see~\cite{asgsch}) we observe that $\Phi_{F}$ satisfies the other properties defining the subspace $V_k^{(n)} (\chi)$.  On the other hand, given $\Phi \in V_k^{(n)}(\chi)$, we can define a function $F$ on $\H_n$ via $F(Z) =
\det(J(g_\infty, iI_n))^{k}\Phi_F(g_\infty)$ where $g_\infty \in \Sp_{2n}(\R)$ is any matrix such that $g_\infty (iI_n) = Z$. Then the properties defining $V_k^{(n)}(\chi)$ imply that $F \in S_k^{(n)}(\chi)$. This completes the proof that the linear mapping $F \mapsto \Phi_F$ is an isomorphism from $S_k^{(n)}(\chi)$ onto $V_k^{(n)}(\chi)$.

Next, let $F \in S_k(\Gamma^{(n)}(N))$ and $p$ any prime. If $p$ does not divide $N$, then the right invariance of $\Phi_F$ by $K_p^N$ shows that $\Phi_F$ is $p$-spherical. Conversely, let $F \in S_k(\Gamma^{(n)}(N))$ and $\Phi_F$ be $p$-spherical. This means that $\Phi_F$ is right-invariant by $G_n(\Z_p)\prod_{q|N, q\neq p}K_q^N \prod_{q \nmid N} G_n(Z_{q})$ which implies that  $F \in S_k(\Gamma^{(n)}(N/(N,p^\infty)))$. In other words, $F$ is $p$-spherical.

To show the isometry, we proceed exactly as in~\cite[p. 188]{asgsch} for each $F \in S_k(\Gamma^{(n),*}(N), \chi)$; the details are omitted.

Finally, to show the Hecke-equivariance, it suffices to show (because of strong approximation) that $\Phi_{F |_k T^*}$ and $\widetilde{T}\Phi_F$ agree on $G_n(\R)^+$. Let $F \in S_k(\Gamma^{(n)}(N))$ and $p \notin N$. Because of Lemma~\ref{heckeiso}, we may assume that $T = \Gamma^{(n)}(1)M\Gamma^{(n)}(1)$ where $M \in \Delta_{p, N}^{(n)}$. By~\cite[Lemma 2.7.7]{miybook}, we may pick matrices $M_i$ in $\Delta_{p, N}^{(n)}$ such that $$\Gamma^{(n)}(N)M\Gamma^{(n)}(N) = \bigsqcup_i \Gamma^{(n)}(N)M_i =  \bigsqcup_i M_i \Gamma^{(n)}(N) $$ and $$G_n(\Z_p)MG_n(\Z_p) = \bigsqcup_i G_n(\Z_p)M_i =  \bigsqcup_i M_i G_n(\Z_p). $$ The Hecke-equivariance now follows directly from the relevant definitions.

\end{proof}

\begin{remark}Given any $\Phi \in V_k^{(n)}$, the unique function $F$ in $S_k^{(n)}$ whose adelization equals $\Phi$ is given by $$F(Z) =
\det(J(g_\infty, iI_n))^{k}\Phi(g_\infty)$$ where $g_\infty \in \Sp_{2n}(\R)$ is any matrix such that $g_\infty (iI_n) = Z$. The Siegel cusp form $F$ is called the \emph{de-adelization} of $\Phi$. If $\Phi$ is fixed by an open compact subgroup $K_\mathfrak{f}$ of $ G_n(\A_\mathfrak{f})$  and we set $\Gamma = G_n(\Q) \bigcap G_n(\R)^+K_\mathfrak{f}$, then the de-adelization $F$ of $\Phi$ belongs to $S_k(\Gamma)$.
\end{remark}

\medskip

Given any unitary Hecke character $\chi$, there is a natural right-regular action of $G_n(\A)$ on  $L^2_0(G_n(\Q)\bs G_n(\A),$ $\chi)$. The space $L^2_0(G_n(\Q)\bs G_n(\A),$ $\chi)$ decomposes discretely under this action into irreducible subspaces. We define an \emph{irreducible cuspidal representation} of $G_n(\A)$ to be an  irreducible subspace of $\oplus_\chi L^2_0(G_n(\Q)\bs G_n(\A),$ $\chi)$ (with the direct sum taken over all unitary Hecke characters $\chi$). So, by our definition\footnote{Note that our definition differs slightly from some other works, such as~\cite{borel-jacquet}.}, an irreducible cuspidal representation (i) comes with a \emph{specific} realization as a subspace of $L^2_0(G_n(\Q)\bs G_n(\A), \chi)$ for some $\chi$, and (ii) is unitary.  If $\pi_1$ and $\pi_2$ are two irreducible cuspidal representations, we use $\pi_1 = \pi_2$ to denote that they are the \emph{same} irreducible subspace while we use the notation $\pi_1 \simeq \pi_2$ to denote that $\pi_1$ and $\pi_2$  are isomorphic as \emph{representations} of $G_n(\A)$ (but their embeddings in $L^2_0(G_n(\Q)\bs G_n(\A), \chi)$ may differ). The multiplicity-one conjecture for $G_n$ (which is still open) predicts that $\pi_1 \simeq \pi_2$ implies $\pi_1 =\pi_2$.

Let $F \in S_k^{(n)}$ and $\Phi_F \in V_k^{(n)}$ be its adelization. Then $\Phi_F$ generates a representation $\pi_F$ of $G_n(\A)$ under the natural right-regular action of $G_n(\A)$ on vectors in $\bigoplus_\chi L^2_0(G_n(\Q)\bs G_n(\A), \chi)$. Any such $\pi_F$ can be written as a direct sum of \emph{finitely} many irreducible cuspidal representations of $G_n(\A)$. The next few propositions deal with various properties that this correspondence between $F$ and $\pi_F$ satisfies. We are particularly interested in the case when $\pi_F$ is itself irreducible. We recall here that any irreducible cuspidal representation $\pi$ of $G_n(\A)$ is factorizable, i.e., we have an isomorphism of representations $\pi \simeq \otimes_v' \pi_v$, where each $\pi_v$ is an irreducible, unitary, admissible representation of $G_n(\Q_v)$.

 Let us briefly describe the representations of $G_n(\R)$ that are relevant for our purposes. For each positive integer $k$, let $\pi_k^{(n)}$ be the irreducible representation of $\Sp_{2n}(\R)$ constructed in~\cite[Sec. 3.5]{asgsch}; it is a lowest weight representation.  Let $\sgn$ denote the sign character from $\R^\times$ to $\pm 1$. It is known that $\pi_k^{(n)}$ extends uniquely to an irreducible representation of $G_n(\R)$ with central character $\sgn^{nk}$. We denote this irreducible representation also by $\pi_k^{(n)}$. If $k>n$, $\pi_k^{(n)}$ is a holomorphic discrete series representation, and if $k=n$, it is a limit of discrete series representations.

\begin{proposition}\label{piinfprop}Let $F \in S_k^{(n)}$ and $ \pi_F$ be the representation of $G_n(\A)$ generated by $\Phi_F$. Let $\pi_F' \simeq \otimes_v'\pi_{F, v}'$ be any irreducible subrepresentation of $\pi_{F} $.  Then, \begin{enumerate}\item $\pi_{F, \infty}' \simeq \pi_k^{(n)}$,
\item The projection\footnote{There is no \emph{canonical} projection from $\pi_F$ to $\pi_F'$ if the latter occurs in the former with multiplicity greater than 1 (although, it is conjectured that this cannot happen); in that case we fix \emph{any} direct decomposition of the isotypic component of $\pi_F'$ into irreducible subspaces (taking one of these equal to the space of $\pi_F'$) and take the resulting projection from $\pi_F$ to $\pi_F'$.} of $\Phi_F$ on the space of $\pi_F'$ maps under the isomorphism $\pi_F' \simeq \pi_{k}^{(n)} \otimes (\otimes_p' \pi_{F, p}')$ to a vector of the form $\phi_\infty \otimes \alpha$ where $\phi_\infty$ is the unique (up to multiple) lowest weight vector in $\pi_k^{(n)}$.\end{enumerate}
\end{proposition}
\begin{proof}The proof is essentially identical to that of~\cite[Thm. 2]{asgsch}. The vector $\Phi_F$ is sent to a vector in $\pi_{F, \infty}'$ with the properties listed in~\cite[Lemma 11]{asgsch}. But there is a unique representation of $G_n(\R)$ containing a vector with this properties, namely the representation $\pi_k^{(n)}$ described above.
\end{proof}

Next, we describe the unramified principal series representations at the finite places. For further details about the facts below, we refer the reader to~\cite{asgsch}. Let $\chi_0$, $\chi_1, \ldots$, $\chi_n$ be unramified characters of $\Q_p^\times$. Each tuple $(t_0,t_1, \ldots t_n) \in (\Q_p^\times)^{n+1}$ gives an element $\mathrm{diag}(t_1, \ldots t_n, t_0t_1^{-1}, \ldots t_0 t_n^{-1})$ of the standard maximal torus in  $G_n(\Q_p)$. Thus, the tuple $(\chi_0$, $\chi_1, \ldots$, $\chi_n)$ defines a character on the Borel subgroup of $G_n(\Q_p)$ which is trivial on the unipotent component. Normalized induction from the Borel yields a representation of $G_n(\Q_p)$ which has a unique subquotient representation containing a $G_n(\Z_p)$-fixed vector (such representations are called spherical). We denote this subquotient representation by $\pi(\chi_0$, $\chi_1, \ldots$, $\chi_n)$. All spherical representations are obtained this way. The isomorphism class of $\pi(\chi_0$, $\chi_1, \ldots$, $\chi_n)$ depends only on the tuple $(\chi_0$, $\chi_1, \ldots$, $\chi_n)$ modulo the action of the Weyl group $W$.

The central character of $\pi(\chi_0$, $\chi_1, \ldots$, $\chi_n)$  equals $\chi_0^2\chi_1 \ldots \chi_n$. The representation $\pi(\chi_0$, $\chi_1, \ldots$, $\chi_n)$ contains, up to multiples, a unique $G_n(\Z_p)$-fixed vector. Putting $b_i = \chi_i(p)$, we see that the representation $\pi(\chi_0$, $\chi_1, \ldots$, $\chi_n)$ is determined by the \emph{Satake parameters} $(b_0, b_1, \ldots, b_n) \in (\C^*)^{n+1}/W$. Thus, we get a one-to-one correspondence between spherical representations of $G_n(\Q_p)$ and the space  $(\C^*)^{n+1}/W$. Note that if $n=1$, our convention for Satake parameters differs from the standard one for $\GL_2$, because we have identified $\Q_p^{2}$ with the maximal torus in $\GL_2(\Q_p)$ in a slightly non-standard manner.  The representations $\pi(\chi_0$, $\chi_1, \ldots$, $\chi_n)$ for which each $|b_i| = 1$ are called tempered unramified principal series representations.

\begin{proposition}\label{pifinprop}Let $F \in S_k^{(n)}$ and $p$ be a prime such that $F$ is $p$-spherical and is an eigenfunction for $T^{(n)}(p)$ and $T_i^{(n)}(p^2)$ ($i=1,\ldots,n$) with eigenvalues $\lambda^{(n)}(F,p)$ and $\lambda_i^{(n)}(F, p^2)$ respectively. Let $ \pi_F$ be the representation of $G_n(\A)$ generated by $\Phi_F$ and $\pi_F' \simeq \otimes_v'\pi_{F, v}'$ be any irreducible cuspidal representation that occurs as a subrepresentation of $\pi_{F} $.  Then \begin{enumerate} \item $\pi_{F, p}'$ is an unramified principal series representation of $G_n(\Q_p)$ whose Satake parameters are determined uniquely by $\lambda^{(n)}(F,p)$ and $\lambda_i^{(n)}(F, p^2)$.
\item The projection of $\Phi_F$ on the space of $\pi_F'$ maps under the isomorphism $\pi_F' \simeq \pi_{F, p}' \otimes (\otimes_{v\neq p}' \pi_{F, v}')$ to a vector of the form $\phi_p \otimes \alpha$ where $\phi_p$ is the unique (up to multiples) spherical vector in $\pi_{F, p}'$.
\end{enumerate}
\end{proposition}
\begin{proof}Because $F$ is $p$-spherical, so is $\Phi_F$; this follows from Theorem~\ref{t:biject}. It follows  by the discussion above that the local component at $p$ of every irreducible subrepresentation of $\pi_F$ is a spherical representation that is determined uniquely from the Hecke eigenvalues of $F$. The second assertion of the Proposition is a direct corollary of the fact that $\Phi$ is $p$-spherical.
\end{proof}

\begin{remark} It is possible to write down explicitly the Satake parameters at $p$ in terms of the Hecke eigenvalues $\lambda^{(n)}(F,p)$ and $\lambda_i^{(n)}(F, p^2)$ (for $i=1, \ldots, n$), and vice-versa. For example, we have \begin{equation}\lambda^{(n)}(F,p) = p^{\frac{n(n+1)}{4}}b_0\sum_{1 \le i_1 <\ldots <i_r \le n}b_{i_1}\ldots b_{i_r},\end{equation} \begin{equation}\label{centralchar}\lambda_n^{(n)}(F,p^2) = b_0^2 b_1\ldots b_n.\end{equation}
\end{remark}

\begin{proposition}\label{propneareq} The following are equivalent:
\begin{enumerate}
\item $F$ is $p$-spherical and an eigenfunction for $\mathcal{H}_{p,1}^{\mathrm{class}, (n)}$ for almost all primes $p$.

\item  If $\pi_F'$ and $\pi_F''$ are two irreducible cuspidal representations both of which occur as subrepresentations of $\pi_{F} $, then $\pi_{F,p}' \simeq \pi_{F,p}''$ for almost all primes $p$.\footnote{This is often referred to in the literature as $\pi_F'$ and $\pi_F''$ being nearly equivalent.}

\end{enumerate}
\end{proposition}
\begin{proof}Suppose that $F$ is an eigenfunction for $\mathcal{H}_{p,1}^{\mathrm{class}, (n)}$ for almost all primes $p$. Then Prop.~\ref{pifinprop} tells us that $\pi_{F,p}' \simeq \pi_{F,p}''$ for all such $p$. Conversely, if the second assertion of the Proposition holds, let $S$ be a set of places including the real place such that if $p \notin S$ then the local components $\pi_{F,p}'$ of the various irreducible subrepresentations $\pi_F'$  of $\pi_{F} $ are all isomorphic and the projections of $\Phi_F$ on these invariant subspaces are all spherical at $p$. It follows that each of these projections is an eigenfunction for $\mathcal{H}_p^{(n)}$ with the same eigenvalue. So $\Phi_F$ itself is an eigenfunction for $\mathcal{H}_p^{(n)}$. By the Hecke equivariance proved in Theorem~\ref{t:biject}, it follows that $F$  is an eigenfunction for $\mathcal{H}_{p,1}^{\mathrm{class}, (n)}$.
\end{proof}
Let $F \in S_k^{(n)}$ be such that $\pi_F$ is \emph{irreducible}.\footnote{Note that this implies that $F \in S_k^{(n)}(\chi)$ where $\chi$ is the central character of $\pi_F$.} Then Prop.~\ref{propneareq} tells us that $F$ is an eigenfunction of the local Hecke algebras for almost all primes. Conversely, let $F \in S_k^{(n)}$ be an eigenfunction of the local Hecke algebras for almost all primes. If $n=1$, then  strong multiplicity one for $\GL(2)$ and Prop.~\ref{propneareq} together imply that $\pi_F$ is  irreducible. But if $n>1$ we cannot make any such claim. However, if $F$ belongs to $S_k(\Sp_{2n}(\Z))$ and is an eigenfunction of the local Hecke algebras for almost all primes, then one can prove that $\pi_F$ is irreducible. This fact (which is not true for $F$ in $S_k(\Gamma^{(n)}(N))$ when $N>1$) was proved by Narita, Pitale and Schmidt in a recent work~\cite{NPS}, and is quite surprising because multiplicity one for such $\pi_F$ is an open problem.

Suppose $F \in S_k^{(n)}$ and $\pi_F$ is irreducible. Then we know that $\pi_{F, \infty}$ is isomorphic to the specific representation $\pi_k^{(n)}$ mentioned in Prop.~\ref{piinfprop}. It is worth asking if every $\pi$ with this property is obtained as $\pi_F$ for some $F\in S_k^{(n)}$. The next result asserts that the answer is yes, at least for $n=1$ and $n=2$. Because of its importance, we state this as a theorem.

\begin{theorem}\label{t:deadelize}Suppose $n \leq 2$. Let $\pi\simeq\otimes_v'\pi_v$ be an irreducible cuspidal representation of $G_n(\A)$ with central character $\chi$ such that $\pi_\infty$ is isomorphic to the representation $\pi_k^{(n)}$ constructed above. Then there exists $F\in S_k^{(n)}(\chi)$ such that \begin{enumerate}
\item $F$ is $p$-spherical and an eigenfunction for $\mathcal{H}_{p,1}^{\mathrm{class}, (n)}$ for all primes $p$ where $\pi_p$ is unramified.

\item $\pi_F = \pi$.

\end{enumerate}

\end{theorem}

\begin{proof}By Theorem~\ref{t:biject}, it suffices to prove that the space $V_{\pi}$ of $\pi$, considered as a subspace of $L^2_0(G_n(\Q)\bs G_n(\A), \chi)$, contains a vector $\Phi$ such that a) $\Phi$ is in $V_k^{(n)}(\chi)$, and b)  $\Phi$ is $p$-spherical and an eigenfunction for $\mathcal{H}_{p}^{(n)}$ for all primes $p$ where $\pi_p$ is unramified.

Fix an isomorphism $\sigma: \pi \rightarrow \otimes_v'\pi_v$. Let $\phi_\infty$ be  the unique (up to multiple) vector in $\pi_\infty \simeq \pi_k^{(n)}$ that satisfies the properties listed in~\cite[Lemma 11]{asgsch}. For all primes $p$ where $\pi_p$ is unramified, let $\phi_p$ be the unique (up to multiple) spherical vector in $\pi_{p}$.

Now, let $p$ be a prime where $\pi_p$ is ramified. To choose $\phi_p$ for such primes, we use a relatively deep fact, namely that there exists a vector $\phi_p$ in the space of $\pi_p$ that is fixed by the subgroup $K_p^{N_p}$ for some integer $N_p$. If $n=1$ this is a well-known result that can be proved as a direct consequence of either the existence and uniqueness of Kirillov models or the theory of local zeta integrals (but contrary to the assertion of Casselman in~\cite{cass73}, it is \emph{not} an immediate consequence of admissibility). If $n=2$, the existence of such a $\phi_p$ follows from the recent work of Roberts and Schmidt~\cite{roberts-schmidt13} on the theory of local Bessel models.

 Now, let $\phi = \otimes_p \phi_p \otimes \phi_\infty$ and let $\Phi = \sigma^{-1}(\phi)$. Then, from construction, $\Phi$ is in $V_k^{(n)}(\chi)$. Moreover, for all unramified $p$, $\phi_p$ lies in the one-dimensional space of $G_n(\Z_p)$-fixed vectors in $\pi_p$; it follows  that $\Phi$ is $p$-spherical and an eigenfunction for $\mathcal{H}_{p}^{(n)}$ at such $p$. Thus, $\Phi$ satisfies the desires properties mentioned at the beginning of this proof.
\end{proof}

\begin{remark}The author expects that an analogue of the above result holds for all $n$; a proof will possibly involve a close look at either the local Fourier-Jacobi model, or some other nice unique model for local representations of $G_n(\Q_p)$. \end{remark}

\subsection{Some results on arithmeticity}\label{s:arithmeticity}We end our treatment of adelization with a brief discussion of arithmeticity. Let $(\sigma, V_\sigma)$ be an irreducible representation of the group $G_n(\Q_p)$ on a complex vector space $V$. Let $\tau$ be any automorphism of $\C$ and $t: V_\sigma \rightarrow V'$ be a $\tau$-linear isomorphism. We define a representation $( \leftexp{\tau}{\sigma}, V_{\leftexp{\tau}{\sigma}})$ of $G_n(\Q_p)$ as follows: set  $V_{\leftexp{\tau}{\sigma}} = V'$ and define the action of $G_n(\Q_p)$ on  $V_{\leftexp{\tau}{\sigma}}$ by $\leftexp{\tau}{\sigma}(g) = t \circ \sigma(g) \circ t^{-1}$. It is clear that the isomorphism class of  $\leftexp{\tau}{\sigma}$ does not depend on the choice of $V'$ or $\tau$.

If $\pi_p$ is an irreducible, admissible representation of $G_n(\Q_p)$ for some prime $p$, then we let $\Q(\pi_p)$ denote the fixed field of the group of automorphisms $\tau$ of $\C$ for which $\leftexp{\tau}{\pi_p} \simeq \pi_p$. If $\pi_\mathfrak{f} \simeq \otimes_p'\pi_p$ is an irreducible, admissible representation of $G_n(\A_\mathfrak{f})$ then we define another irreducible admissible representation $\leftexp{\tau}{\pi_\mathfrak{f}}$ of $G_n(\A_\mathfrak{f})$ by $\leftexp{\tau}{\pi_\mathfrak{f}} \simeq \otimes'_p \leftexp{\tau}{\pi}_p $. This gives a left action of $\Aut(\C)$ on the set of isomorphism classes of irreducible, admissible representations of $G_n(\A_\mathfrak{f})$. We let $\Q(\pi_\mathfrak{f})$ denote the fixed field of the group of automorphisms $\tau$ of $\C$ for which $\leftexp{\tau}{\pi_\mathfrak{f}} \simeq \pi_\mathfrak{f}$. Clearly $\Q(\pi_\mathfrak{f}) = \coprod_p \Q(\pi_p).$ If $\pi \simeq \pi_\infty \otimes \pi_\mathfrak{f}$ is an irreducible, admissible representation of $G_n(\A)$ then we define another irreducible admissible representation $\leftexp{\tau}{\pi}$ of $G_n(\A)$ by $\leftexp{\tau}{\pi} \simeq \pi_\infty \otimes\leftexp{\tau}{\pi_\mathfrak{f}} $.

\begin{proposition}\label{p:blaharam}Let $\pi \simeq \otimes_v'\pi_v$ be an irreducible cuspidal representation of $G_n(\A)$ such that $\pi_\infty$ is isomorphic to the representation $\pi_k^{(n)}$ constructed above. Assume that $k \ge n$ and $nk$ is even. Then
\begin{enumerate}
\item $\Q(\pi_\mathfrak{f})$ is a finite extension of $\Q$ and is contained in $\Q^{\CM}$.
\item For any automorphism $\tau$ of $\C$, $\tau (\Q(\pi_\mathfrak{f})) = \Q(\leftexp{\tau}\pi_\mathfrak{f}).$

\item For any automorphism $\tau$ of $\C$, there exists an  irreducible cuspidal representation of $G_n(\A)$ that is isomorphic to $\leftexp{\tau}{\pi}$. Moreover, if $F \in S_k^{(n)}$ is such that $\pi_F = \pi$, then we have $$\pi_{{}^\tau\!F} \simeq  \leftexp{\tau}{\pi},$$ where ${}^\tau\!F$ is defined as in~\eqref{autcact}.
\end{enumerate}
\end{proposition}
\begin{proof}The first assertion follows from~\cite[Thm. 4.4.1]{blaharam}. The second assertion is trivial. The last assertion follows from~\cite[Thm. 4.2.3]{blaharam} where the existence of a certain irreducible cuspidal representation $\leftexp{(\tau)}{\pi}$, isomorphic to $\leftexp{\tau}{\pi}$, is asserted for each $\tau \in \Aut(\C)$. The representation $\leftexp{(\tau)}{\pi}$ is constructed by viewing the holomorphic vectors in the space of $\pi$ as sections of certain holomorphic vector bundles and defining an action of $\Aut(\C)$ on these sections; see also~\cite[Section 3]{morimoto} for a fairly explicit description for $n=2$ (which is the only case needed for the main results of this paper). In the special case $\pi=\pi_F$, when one takes the holomorphic vector to be (an appropriate multiple of) $\Phi_F$, then the action of $\Aut(\C)$ on this vector is compatible with that on the Fourier coefficients of $F$; this follows from~\cite[(1.1.3)]{blaharam} (see also~\cite[equation (34)]{morimoto} for an explicit expression when $n=2$). This implies that $\pi_{{}^\tau\!F} \simeq \leftexp{(\tau)}{\pi}_F$ since $\Phi_F$ generates the irreducible representation $\pi_F$. Since we already have $\leftexp{(\tau)}{\pi}_F \simeq \leftexp{\tau}{\pi}_F$, it follows that $\pi_{{}^\tau\!F} \simeq \leftexp{\tau}{\pi}_F.$
\end{proof}

\begin{remark} The above proposition gives a left action of $\Aut(\C)$ on the set of \emph{isomorphism classes} of irreducible cuspidal representations of $ G_n(\A)$ whose component at infinity is equal to $\pi_k^{(n)}$ with $k \ge n$ and $kn$ even.
\end{remark}

Let $S$ be any subset of the finite places (i.e., primes) of $\Q$. For any irreducible admissible representation $\pi_S=\otimes_{p\in S}'\pi_p$ of $G_n(\A_S) = \otimes_{p\in S}' G_n(\Q_p)$, we denote $\Q(\pi_{S}) = \coprod_{p \in S}\Q(\pi_p)$ (put $\Q(\pi_{\emptyset}) = \Q)$ and we let $\mathfrak{f}$ denote the set of all finite primes; this makes the above definition consistent with the earlier one for $\Q(\pi_\mathfrak{f})$.  For any congruence subgroup $\Gamma$ of $G_n(\Q)$, and any admissible $\pi_S$ as above, we  define
$S_k(\Gamma;  \pi_S)$ to be equal to the subspace of $S_k(\Gamma)$ consisting of $\{0\}$ and all those $F$ which have the property that $\pi'_{F,p} \simeq \pi_p$  for each irreducible subrepresentation $\pi'_F$  of $\pi_F$  and each $p\in S.$ Note that whenever the Dirichlet density of $S$ is at least $1/2$, the space $S_k(\Gamma^{(n)}(N); \pi_S)$ is contained inside  $S_k(\Gamma^{(n), *}(N), \chi)$ for some $\chi$ (with $\chi$ depending on $\pi_S$); this follows from Proposition~\ref{funprop}. By Proposition~\ref{p:blaharam}, it follows that $\Q(\pi_{S})$ is a subfield of a CM field whenever $S_k(\Gamma^{(n)}(N); \pi_S) \neq 0$ for some $k$, $N$, $n$ with $k \ge n$ and $nk$ even.
 For any subfield $L$ of $\C$, we denote $S_k(\Gamma; L, \pi_S) = S_k(\Gamma;  \pi_S) \cap S_k(\Gamma; L)$.

 \begin{remark}\label{remarkneareq}Note that Proposition~\ref{propneareq} can be rephrased as follows: $F \in S_k(\Gamma^{(n)}(N))$ is an eigenfunction of the local Hecke algebras at almost all primes if and only if $F$ belongs to $S_k(\Gamma^{(n)}(N); \pi_S)$ for some subset $S$ of the primes containing almost all primes and some irreducible admissible representation $\pi_S$ of $G_n(\A_S)$.
 \end{remark}

\begin{lemma}\label{lpeterssonequiva}Let $S$ be a subset of the primes containing almost all primes and $\pi_S=\otimes_{p\in S}'\pi_p$ be an irreducible admissible representation of $G_n(\A_S) = \otimes_{p\in S}' G_n(\Q_p)$. Let $N$, $k$, $n$ be positive integers such that $k \ge n$ and $nk$ is even. Suppose that $F \in S_k(\Gamma^{(n)}(N); \Q(\pi_S), \pi_S)$ and $G \in S_k(\Gamma^{(n)}(N))$. Then for all $\sigma \in \Aut(\C)$, we have \begin{equation}\label{peterssonequiva} \sigma \left( \frac{\langle G, F\rangle}{\langle F,F \rangle}  \right) =  \frac{\langle \leftexp{\sigma}G, \leftexp{\sigma}F\rangle}{\langle \leftexp{\sigma}F,\leftexp{\sigma}F \rangle}.
\end{equation}

\end{lemma}
\begin{proof} We consider two cases. First, assume $\pi_S \simeq \bar{}\pi_S$ where $\bar{}$ denotes complex conjugation. In this case, the field $\Q(\pi_S)$ and hence the Fourier coefficients of $F$ are totally real; so~\eqref{peterssonequiva} follows from the main result of~\cite{gar2}. Next, we consider the case where $\pi_S \not \simeq \bar{}\pi_S$. In this case, we note that $\langle F, \bar{F}\rangle =0$ (since their adelizations give rise to disjoint subsets of irreducible cuspidal representations). Put $F_1 = (F + \bar{F})/2$ and $F_2 = (F - \bar{F})/(2i)$. Thus, $F_1$ and $F_2$ have totally real algebraic Fourier coefficients. Moreover, we have $\langle F, F\rangle = 2 \langle F_1, F_1 \rangle = 2 \langle F_2, F_2 \rangle$, and $\langle F_1, F_2 \rangle = 0$. Finally  $F_1$ and $F_2$ are both eigenfunctions of the local Hecke algebra for the subgroup $\Sp_{2n}$ of $G_n$ at almost all places (with the same Hecke eigenvalues); this follows from~\eqref{heckeherm}. Now the result follows from applying the main result of~\cite{gar2} to $F_1$ as well as to $F_2$ and using that $F= F_1 + i F_2$.
\end{proof}

The following elementary lemma (which is basically Hilbert 90 for $\GL_n$) will be used in the proof of Proposition~\ref{rationalbasis}.

\begin{lemma}\label{lemma:linear}Let $V$ be a subspace of $\C^n$ and $L$ a subfield of $\C$. Then, the following conditions are equivalent.
\begin{enumerate}
\item $V$ has a basis $\{v_1,\ldots,v_m\}$ with each $v_i \in V \cap L^n$.
   \item $V$ is invariant under the action of $\Aut(\C/L)$.
\end{enumerate}
\end{lemma}
\begin{proof}If $V$ has a basis comprising of elements of $L^n$ then it is clear that $V$ is invariant under the action of $\Aut(\C/L)$. To prove the converse, let us start with any subspace $V$ of $\C^n$. By Gauss-Jordan elimination, we can find a basis of $V$ such that the basis matrix is in reduced row-echelon form. Concretely, this means that $V$ has a basis $\{v_1, v_2,\ldots,v_m\}$ where if we write $v_i = (v_{i,j})_{j=1,\ldots,n}$ then the matrix $B = (v_{i,j})$ has the following properties:
\begin{itemize}
\item The leading coefficient (first non-zero number from the left) of a row is always strictly to the right of the leading coefficient of the row above it.
\item Every leading coefficient is 1 and is the only nonzero entry in its column.
\end{itemize}
Let $1\le i \le m$ and $\tau \in \Aut(\C/L)$. Then by assumption, $\tau(v_i) = \sum_{j=1}^m c_j v_j$ for some complex numbers $c_j$. Looking at the leading coefficient of $v_i$ we see that $c_i=1$. Moreover, if $c_j \neq 0$ for some $j\neq i$, then by considering the leading coefficient of the $j$th row we deduce that $c_j= 0.$ Thus $\tau(v_i) = v_i$ for all $\tau \in \Aut(\C/L)$. Since $\tau$ was arbitrary and $i$ was any integer between 1 and $m$, it follows that all entries of $B$ lie in $L$. This completes the proof that the basis $\{v_1,\ldots,v_m\}$ of $V$ satisfies the property that each $v_i \in V \cap L^n$.
\end{proof}

\begin{proposition}\label{rationalbasis}Let $S$ be any set of primes and $\pi_S=\otimes_{p\in S}'\pi_p$ be an irreducible admissible representation of $G_n(\A_S) = \otimes_{p\in S}' G_n(\Q_p)$. Let $N$, $k$, $n$ be positive integers such that $k \ge n$ and $nk$ is even. Then there is a basis  of $S_k(\Gamma^{(n)}(N); \pi_S)$ consisting of cuspforms lying in $S_k(\Gamma^{(n)}(N);$ $\Q(\pi_S),  \pi_S)$.
\end{proposition}

\begin{proof}Let $r$ denote the dimension of $S_k(\Gamma^{(n)}(N))$. Put $L=\Q(\pi_S)$ and let $L^r$ be the subset of $\C^r$ consisting of those vectors all of whose entries lie in $L$. Note that $\Aut(\C)$ acts on the set $\C^r$ by its action on each entry.

We identify $S_k(\Gamma^{(n)}(N))$ with $\C^r$ by picking a basis of forms with rational Fourier coefficients (this is possible by~\eqref{sh:rationalbasis}). Under this identification, $S_k(\Gamma^{(n)}(N);L)$ corresponds to $L^r$. Let $V$ be the subset of $L^r$ corresponding to $S_k(\Gamma^{(n)}(N); L, \pi_S)$ under the above identification. By Proposition~\ref{p:blaharam} and the linearity of the adelization map, it follows that $V$ is a vector subspace of $L^r$ and the action of $\Aut(\C / L)$ leaves the space $V$  invariant. It follows from Lemma~\ref{lemma:linear}  that $V$ has a basis consisting of vectors in $L^r$; in other words, that $S_k(\Gamma^{(n)}(N); \pi_S) = S_k(\Gamma^{(n)}(N); \Q(\pi_S), \pi_S) \otimes \C.$ This proves the proposition.
\end{proof}

We now state our main result on arithmeticity.\footnote{In the literature, one can find several results similar to Theorem~\ref{t:aithmeticity}; however they all appear to be weaker or less general than our formulation in some respect or another.}
\begin{theorem}\label{t:aithmeticity}Let $N$, $k$, $n$ be positive integers such that $k \ge n$ and $nk$ is even. Let $S$ be a subset of the primes containing almost all primes and let $\mathfrak{S} = \mathfrak{S}_S $ be the set of isomorphism classes of irreducible admissible representations $\pi_S$ of $G_n(\A_S)$ such that $S_k(\Gamma^{(n)}(N); \pi_S) \neq \{0\}.$  Then there exists a direct sum decomposition into orthogonal subspaces $$S_k(\Gamma^{(n)}(N)) = \oplus_{\pi_S \in \mathfrak{S}} S_k(\Gamma^{(n)}(N); \pi_S),$$ and moreover, for each $\pi_S \in \mathfrak{S}$, there is an orthogonal basis $\{F_1^{(\pi_S)}, \ldots, F_{r_{\pi_S}}^{(\pi_S)} \}$ of $S_k(\Gamma^{(n)}(N); \pi_S)$ (where  $r_{\pi_S}$ denotes the dimension of $S_k(\Gamma^{(n)}(N); \pi_S)$) with the following properties:
\begin{enumerate}
\item $F_t^{(\pi_S)} \in S_k(\Gamma^{(n)}(N); \Q( \pi_S), \pi_S)$ for each $\pi_S \in \mathfrak{S}$, and $1 \le t \le r_{\pi_S}$.
 \item Let $\sigma \in \Aut(\C)$, $\pi_S \in \mathfrak{S}$. Then $r_{\leftexp{\sigma}\pi_S} = r_{\pi_S}$ and for each $1 \le t \le r_{\pi_S}$, we have $\leftexp{\sigma} F_t^{(\pi_S)} = F_t^{\leftexp{\sigma}\pi_S}.$

\item If we express $F \in S_k(\Gamma^{(n)}(N))$  in terms of the basis elements $F_t^{(\pi_S)}$: $$F = \sum_{\pi_S \in \mathfrak{S}} \sum_{t=1}^{r_{\pi_S}} a(F,  F_t^{(\pi_S)}) F_t^{(\pi_S)}, \quad a(F,  F_t^{(\pi_S)}) \in \C,$$ then we have $\sigma (a(F,  F_t^{(\pi_S)})) = a(\leftexp{\sigma}F,  F_t^{(\leftexp{\sigma}\pi_S)})$ for all $\sigma \in \Aut(\C)$. In particular, if $F \in S_k(\Gamma^{(n)}(N); L)$, then $a(F,  F_t^{(\pi_S)}) \in \Q(\pi_S)L.$

\item  For each $\pi_S \in \mathfrak{S}$, $1 \le t \le u \le r_{\pi_S}$, and $\sigma \in \Aut(\C)$, we have \begin{equation}\label{petequiv2}\sigma \left( \frac{\langle F_t^{(\pi_S)}, F_t^{(\pi_S)} \rangle}{\langle F_u^{(\pi_S)}, F_u^{(\pi_S)} \rangle}  \right) =  \frac{\langle F_t^{(\leftexp{\sigma}\pi_S)}, F_t^{(\leftexp{\sigma}\pi_S)} \rangle}{\langle F_u^{(\leftexp{\sigma}\pi_S)}, F_u^{(\leftexp{\sigma}\pi_S)} \rangle}.\end{equation}

\end{enumerate}

\end{theorem}
\begin{proof}The direct sum decomposition into orthogonal subspaces as stated is clear from the definitions, and the fact that vectors in non-isomorphic irreducible cuspidal representations are orthogonal with respect to the Petersson inner product.

Now, pick any $\pi_S \in \mathfrak{S}$ and let $\{F_1^{(\pi_S)}, \ldots, F_{r_{\pi_S}}^{(\pi_S)} \}$ be a basis of $S_k(\Gamma^{(n)}(N); \pi_S)$ consisting of forms lying in $S_k(\Gamma^{(n)}(N);$ $\Q(\pi_S),  \pi_S)$; this is possible by Proposition~\ref{rationalbasis}. Then $\{\leftexp{\sigma}F_1^{(\pi_S)}, \ldots, \leftexp{\sigma}F_{r_{\pi_S}}^{(\pi_S)} \}$ is a basis of $S_k(\Gamma^{(n)}(N); \leftexp{\sigma}\pi_S)$ consisting of forms lying in $S_k(\Gamma^{(n)}(N);$ $\Q(\leftexp{\sigma}\pi_S),  \leftexp{\sigma}\pi_S)$. We now apply the
Gram-Schmidt operation to these bases. Using Lemma~\ref{lpeterssonequiva}, this gives us the existence of an orthogonal basis of the desired sort for $S_k(\Gamma^{(n)}(N); \pi'_S)$ for all conjugates $\pi_S' = \leftexp{\sigma}\pi_S$ of $\pi_S$ in $\mathfrak{S}$.

We now repeat the above process, as many times as necessary, picking at each stage some $\pi_S \in \mathfrak{S}$ that has not been dealt with so far. This completes the proof of the second claim.

The third claim follows directly from Lemma~\ref{lpeterssonequiva}. The final claim follows also from Lemma~\ref{lpeterssonequiva} by putting $G= F_t^{\leftexp{\sigma}\pi_S} + F_u^{\leftexp{\sigma}\pi_S}$, $F=F_t^{\leftexp{\sigma}\pi_S}$ in~\eqref{peterssonequiva}.
\end{proof}

As an easy corollary of Theorem~\ref{t:aithmeticity}, we have the following strengthening of Lemma~\ref{lpeterssonequiva}.
\begin{corollary}\label{corpeteq}Let $k$, $n$ be positive integers such that $k \ge n$ and $nk$ is even. Let $F, G \in S_k^{(n)}$.
\begin{enumerate}
\item Suppose that $F$ is an eigenfunction of the local Hecke algebras  $\mathcal{H}_{p,1}^{\mathrm{class}, (n)}$ for almost all primes $p$. Assume that there is a constant $c$ such that all the Fourier coefficients of $cF$ lie in $\Q^{\CM}$.\footnote{Note that this implies that the Fourier coefficients of $cF$ generate a number field that is either totally real or a $\CM$ field.} Then for all $\sigma \in \Aut(\C)$, we have $$ \sigma \left( \frac{\langle G, F\rangle}{\langle F,F \rangle}  \right) =  \frac{\langle \leftexp{\sigma}G, \leftexp{\sigma}F\rangle}{\langle \leftexp{\sigma}F,\leftexp{\sigma}F \rangle}.
$$

\item Suppose that for almost all primes $p$, $F$ and $G$ are eigenfunctions of the local Hecke algebras  $\mathcal{H}_{p,1}^{\mathrm{class}, (n)}$ with the same eigenvalues. Assume that there is a constant $c$ such that all the Fourier coefficients of $cF$ and $cG$ are in $\Q^{\CM}$. Then for all $\sigma \in \Aut(\C)$, we have $$ \sigma \left( \frac{\langle G, G\rangle}{\langle F,F \rangle}  \right) =  \frac{\langle \leftexp{\sigma}G, \leftexp{\sigma}G\rangle}{\langle \leftexp{\sigma}F,\leftexp{\sigma}F \rangle}.$$ In particular,  $$\frac{\langle G, G\rangle}{\langle F,F \rangle} \in \Q^{\CM}.$$

\end{enumerate}

\end{corollary}
\begin{proof}Both parts follow immediately by writing $F$ and $G$ in terms of the basis (take $S$ to be the set of all primes where $F$ and $G$ are $p$-spherical and Hecke eigenfunctions with equal eigenvalues) obtained in Theorem~\ref{t:aithmeticity}, and then using Remark~\ref{remarkneareq} and the fact that $\overline{\sigma(x)} = \sigma{(\overline{x})}$ for all $x \in \Q^{\CM}$, $\sigma \in \Aut(\C)$.
\end{proof}

\begin{remark}Corollary~\ref{corpeteq} can be regarded as a gentle strengthening of the main result in~\cite{gar2}.
\end{remark}

\begin{remark}Corollary~\ref{corpeteq} is not true without the assumption about the Fourier coefficients lying in $\Q^{\CM}$. Counterexamples can be easily obtained by considering forms like $F_1 + b F_2$ where $F_1$ and $F_2$ are linearly independent forms satisfying the assumption and $b \notin \Q^{\CM}$.
\end{remark}

\section{Yoshida lifts, Weak Yoshida lifts and arithmetic properties}\label{s:yoshida}

\subsection{Yoshida lifts and Yoshida spaces}\label{s:yoshidarep}Let $v$ be a place of $\Q$. Let $\Phi(G_n(\Q_v))$ denote the set of isomorphism classes of admissible homomorphisms $\phi:W_v' \rightarrow \leftexp{L}G_n^0$ where $W_v'$ denotes the Weil-Deligne group of $\Q_v$ and  $\leftexp{L}G_n^0$ denotes the dual group of $G_n$. The elements of  $\Phi(G_n(\Q_v))$ are called $L$-parameters  for $G_n(\Q_v).$ The dual group $\leftexp{L}G_n^0$  of $G_n$ is in general an orthogonal group but if $n=1$ or $n=2$ then the dual group (by accidental isomorphism) happens to be a symplectic group. More precisely, $\leftexp{L}G_1^0 = \GL_2(\C)= G_1(\C)$ and $\leftexp{L}G_2^0 = \GSp_4(\C) = G_2(\C)$. Thus, if $n \le 2$, then it makes sense to talk of the multiplier (similitude character) $\mu_n$ of an element of $\leftexp{L}G_n^0$.

Let $\Pi(G_n(\Q_v))$ denote the set of isomorphism classes of irreducible admissible representations of $G_n(\Q_v)$. Assume henceforth that $n=1$ or $n=2$. Then the local Langlands conjecture is known; see~\cite{knapplocal, bush-hen, gantakGSp4}. Thus, we have a finite-to-one surjective map
\begin{equation}\label{locallanglands}L: \Pi(G_n(\Q_v)) \rightarrow \Phi(G_n(\Q_v)), \quad n=1 \text{ or }2,\end{equation} satisfying certain conditions. Given any irreducible admissible representation $\pi_v$ of $G_n(\Q_v)$, we call $L(\pi_v)$ the $L$-parameter of $\pi_v$. Two irreducible admissible representations have the same central character if and only if their $L$-parameters have the same multiplier (similitude character) $\mu_n$.

We have an embedding of dual groups \begin{align}\label{dualgroupmorphismeq2}
 \{(g_1,g_2)\in G_1(\C)\times G_1(\C)\:|\:\mu_1(g_1)=\mu_1(g_2)\}&\longrightarrow G_2(\C),\\
 (\mat{a}{b}{c}{d},\mat{a'}{b'}{c'}{d'})&\longrightarrow\left(\begin{matrix}a&&b\\&a'&&b'\\c&&d\\&c'&&d'\end{matrix}\right).\nonumber
\end{align}

By abuse of notation, we identify the pair $(g_1, g_2)$ with its image in $G_2(\C)$ under the above embedding. If $\phi_1$ and $\phi_2$ are $L$-parameters for $G_1(\Q_v)$ with the same similitude character, then we define their direct sum $\phi_1 \oplus \phi_2$ to be the $L$-parameter for $G_2(\Q_v)$ given by
$$(\phi_1 \oplus \phi_2)(w) = (\phi_1(w), \phi_2(w)).$$ The following key result is essentially due to Brooks Roberts~\cite{rob2001}; see also the discussion in~\cite{gantakGSp4} and~\cite[Sec. 3.2]{sahaschmidt}.

\begin{theorem}[Roberts~\cite{rob2001}]\label{t:roberts} Let $\pi_1 = \otimes'_v \pi_{1,v}$, $\pi_2=\otimes'_v \pi_{2,v}$ be irreducible cuspidal representations of $G_1(\A)$. Assume that $\pi_1$ and $\pi_2$ are non-isomorphic, tempered everywhere, have the same central character $\chi$, and that there exist integers $k_1 \ge k_2\ge 2$ such that $\pi_{1, \infty} \simeq \pi_{k_1}^{(1)}$ and $\pi_{2, \infty} \simeq \pi_{k_2}^{(1)}.$ Let $T$ be the set of finite places $p$ where $\pi_{1,p}$ and $\pi_{2,p}$ are both discrete series. Then,
\begin{enumerate}
\item There exists an irreducible cuspidal representation $\pi = \otimes_v' \pi_v$ of $G_2(\A)$ with central character $\chi$ such that \begin{equation}\label{lparsum}L(\pi_v) = L(\pi_{1,v}) \oplus L(\pi_{2,v}) \quad \text{ for all places $v$.}\end{equation} Any such $\pi$ occurs with multiplicity one in $L^2_0(G_2(\Q)\bs G_2(\A), \chi)$. Moreover, the number of distinct
    $\pi$ as above equals $2^{\#T}$.

\item There  exists an irreducible cuspidal representation $\pi = \otimes_v' \pi_v$ of $G_2(\A)$ satisfying~\eqref{lparsum} and such that $\pi_\infty \simeq \pi_k^{(2)}$ for some $k$ if and only if $T$ is non-empty and $k_2=2$. In that case, we must have $k = k_1/2 + 1$.\footnote{Note that $k_2=2$ automatically implies that $k_1$ is even because $f$ and $g$ have the same nebentypus.} The set of primes $p$ at which $\pi_p$ is non-generic is a subset $S_{\pi}$ of $T$ with an odd number of elements. Moreover, the number of distinct $\pi$ as above equals $2^{\#T -1}$, corresponding to the subsets of $T$ with odd cardinality.
\end{enumerate}
\end{theorem}

We now define the notion of a Yoshida lift. Let $ f  $, $g$ be classical holomorphic newforms of some weights, levels and characters. Then, it is well-known that the representations $\pi_f$, $\pi_g$ are irreducible (indeed,  the stated irreducibility is true whenever $f$ and $g$ are Hecke eigenforms almost everywhere; this follows from Proposition~\ref{propneareq} and the comments immediately after). Assume that the characters of $f$ and $g$ are associated with the same primitive character (equivalently, $\pi_f$, $\pi_g$ have the same central character) $\chi$.

\begin{definition}\label{defyosh}Given two classical holomorphic newforms $f$, $g$ whose characters are associated with the same primitive character $\chi$,  a non-zero element $F\in S_k^{(2)}$ is said to be a Yoshida lift of $(f, g)$ if $\pi_F$ is irreducible and    \begin{equation}\label{lparsum2v}L(\pi_{F,v}) = L(\pi_{f,v}) \oplus L(\pi_{g,v}) \quad \text{ for all places $v$.}\end{equation}
\end{definition}
Note that any Yoshida lift $F$ of $(f, g)$ must be an element of $S_k^{(2)}(\chi)$ and must also be an eigenfunction of $\mathcal{H}_{p,1}^{\mathrm{class}, (2)}$ for almost all $p$; the latter fact follows from Proposition~\ref{propneareq}.

\begin{remark}\label{yoshbasicrem}Let $f$, $g$ be newforms and $F\in S_k^{(2)}$ be a Yoshida lift of $(f,g)$. Then, for any irreducible cuspidal representation $\pi$ of $\GL_n(\A)$, we have an equality of local $L$-functions at each place $v$,
\begin{equation}L(s,  \pi_v\times\pi_{F,v} ) = L(s, \pi_{v} \times \pi_{f,v})L(s, \pi_{v} \times \pi_{g,v})
\end{equation} where the local factors are defined using the local Langlands correspondence (if $\pi_{F,v}$ is either generic or non-supercuspidal, these local factors can be also defined using the method of Shahidi).

\end{remark}

Combining Theorem~\ref{t:roberts} and Theorem~\ref{t:deadelize}, we immediately deduce:

\begin{theorem}\label{t:yoshida}Let $ f$, $g$ be classical holomorphic newforms of weights $k_1$, $k_2$ respectively and characters $\chi_1$, $\chi_2$ respectively. Suppose that \begin{enumerate} \item $k_1 \ge k_2 \ge 2,$ \item $\chi_1$ and $\chi_2$ are associated with the same primitive character $\chi$, \item $f$ is not a multiple of $g$. \end{enumerate} Then there exists a Yoshida lift of $(f, g)$ if and only if both the following conditions are met: a) $k_2 = 2$, b) $\pi_{f,p}$, $\pi_{g,p}$ are both discrete series at some prime $p$. Any such lift $F$ belongs to $S^{(2)}_{k_1/2 + 1}(\chi)$. Moreover, if $T$ is the set of finite places $p$ where $\pi_{1,p}$ and $\pi_{2,p}$ are both discrete series, then all the possible Yoshida lifts $F$ of $(f,g)$ give rise to exactly $2^{\#T -1}$ distinct cuspidal representations $\pi_F$ which are all non-isomorphic to each other.
\end{theorem}

\begin{remark} The above Theorem gives an excellent demonstration of the power of Theorem~\ref{t:deadelize}. Roughly speaking, Theorem~\ref{t:deadelize} enables us to use existence results for automorphic representations to deduce existence of Siegel cusp forms with specified Hecke eigenvalue properties \emph{without needing to make any explicit constructions}. The disadvantage of this approach is that it does not provide us with delicate information about the Fourier coefficients or how to properly normalize the cusp form without doing significant additional work (these points can become important if one cares about integrality questions, congruence primes, etc.).
\end{remark}
\begin{definition}\label{d:yosh}Given two classical holomorphic newforms $f$ and $g$, we say that $f$ and $g$ satisfy the conditions for a scalar valued Yoshida lifting if \begin{enumerate}
\item $f$ is not a multiple of $g$.
\item The characters  of $f$ and $g$ are associated with the same primitive character.
\item One of the weights is equal to 2, and the other is an even integer greater than or equal to 2.
 \item  There exists a (finite) prime $p$ where $\pi_{f,p}$ and $\pi_{g,p}$ are both discrete series.

\end{enumerate}
\end{definition}

\begin{remark}\label{yoshbasicrem2} Let $ f \in S_{k_1}(N_1, \chi_1)$, $g \in S_{k_2}(N_2, \chi_2)$ be classical holomorphic newforms such that $\chi_1, \chi_2$ are associated to the same primitive character $\chi$. Let $p$ be a prime not dividing $N_1$ or $N_2$. Then, $\pi_{f, p}$ and $\pi_{g, p}$ are both unramified principal series representations whose Satake parameters can  be written down easily in terms of the Hecke eigenvalues $\lambda^{(1)}(f,p)$ and $\lambda^{(1)}(g,p)$. There is a unique representation $\Pi_p$ of $G_2(\Q_p)$ satisfying \begin{equation}\label{e:isobaricsumatp}L(\Pi_{p}) = L(\pi_{f,p}) \oplus L(\pi_{g,p});\end{equation} this representation is an unramified principal series representation with Satake parameters related to those of $\pi_{f,p}$ and $\pi_{g,p}$ in a simple manner. Let us write down what this relation means for the Hecke eigenvalues at $p$ of a Yoshida lift. Assume that $f$, $g$ satisfy the conditions for a scalar valued Yoshida lifting (with $k_2 = 2$) and let  $F \in S_{k_1/2 + 1}(\chi)$ be a Yoshida lift of $f, g$. Let $p$ be a prime such that $f$, $g$, $F$ are all $p$-spherical. Then, $\pi_{F, p} \simeq \Pi_p$ where $\Pi_p$ is as in~\eqref{e:isobaricsumatp}. The above discussion and the calculations in~\cite{andrianov} now imply that  $F$ is an eigenfunction for $T^{(2)}(p)$ and $T_i^{(2)}(p^2)$ ($i=1,2$) with eigenvalues $\lambda^{(2)}(F,p)$ and $\lambda_i^{(2)}(F, p^2)$ ($i=1,2$) given as follows:
$$\lambda^{(2)}(F,p) = p(\lambda^{(1)}(f,p) + \lambda^{(1)}(g,p)), \qquad  \lambda_2^{(2)}(F, p^2) = \chi(p),$$ $$ \lambda_1^{(2)}(F, p^2) = \chi(p)(p^2-1) + p\lambda^{(1)}(f,p)\lambda^{(1)}(g,p).$$

\end{remark}

\medskip

We now explain the relation between Yoshida lifts as defined in this paper, and scalar valued classical Yoshida lifts of squarefree level as defined, for instance, in~\cite{bocsch, bocsch1991, bocsch1994, bocsch1997, sahaschmidt}. Let us first recall the latter construction. For any positive integer $N$,
define
\begin{equation}\label{Gamma0defeq0}
\Gamma_0^{(2)}(N) := \left\{\begin{pmatrix}A&B\\ C&D \end{pmatrix} \in \Sp_4(\Z)\;|\;C \equiv 0 \pmod{N}\right\}.
\end{equation}

 Let $k_1 \ge 2$ be an even integer and $N_1$, $N_2$ be two positive, squarefree integers such that $M = \gcd(N_1, N_2)>1$. Let $f$ be a classical newform of weight $k_1$, level $N_1$, trivial nebentypus and $g$ be a classical newform of weight $2$, level $N_2$, trivial nebentypus such that $f$ and $g$ are not multiples of each other. Assume that for all primes $p$ dividing $M$ the Atkin-Lehner eigenvalues of $f$ and $g$ coincide. Put $N = \mathrm{lcm}(N_1, N_2)$. By a result from~\cite{sahaschmidt}, or equivalently, by the more classical constructions of~\cite{bocsch, bocsch1991, bocsch1994, bocsch1997}, there exists for each divisor $M_1$ of $M$ with an \emph{odd} number of prime factors a certain non-zero holomorphic Siegel cusp form $F_{f,g;M_1} \in S_{\frac{k_1}2+1}(\Gamma^{(2)}_0(N))$ that is a Hecke eigenform and whose spin $L$-function is the product of the $L$-functions of $f$ and $g$. We shall refer to the forms $F_{f,g;M_1}$ as the classical Yoshida lifts attached to the pair $(f,g)$. It can be shown that each $F_{f,g;M_1}$  is a Yoshida lift of $(f,g)$ in the sense of \emph{this} paper (the relevant conditions from Definition~\ref{defyosh} can be easily seen to hold using the arguments of~\cite[Section 3.3]{sahaschmidt}); in particular, their adelizations generate irreducible representations. If $\Pi= \Pi_{f,g;M_1}$ is the irreducible cuspidal representation generated by the adelization of $F_{f,g;M_1}$, then $\Pi_v$ is non-generic exactly at the primes dividing $M_1$ and at infinity. If $D$ be the definite quaternion algebra over $\Q$ ramified exactly at (infinity and) the primes dividing $M_1$, then $F_{f,g;M_1}$ can be constructed explicitly by relating $f$, $g$ to forms $f'$ and $g'$ on $D^\times$ via the Jacquet-Langlands-Shimizu correspondence, and then taking a  global theta lift from $D^\times \times  D^\times$ to $\GSp_4$. One can write down the Fourier coefficients of $F_{f,g;M_1}$ precisely in terms of certain representation numbers.

We now prove that a Yoshida lift $F$ as defined in this paper is equal to (a multiple of) one of the classical Yoshida lifts $F_{f,g;M_1}$ whenever $F$ is defined with respect to a Siegel congruence subgroup of squarefree level.

\begin{proposition}\label{equivclassical}Let $k_1 \ge 2$ be an even integer and $N_1$, $N_2$ be two positive, squarefree integers. Put $N = \mathrm{lcm}(N_1, N_2)$. Let $f$ be a classical newform of weight $k_1$, level $N_1$, trivial nebentypus and $g$ be a classical newform of weight $2$, level $N_2$, trivial nebentypus such that $f$ and $g$ are not multiples of each other. Then there exists a Yoshida lift $F$ of $(f,g)$ such that $F \in S_{\frac{k_1}2+1}(\Gamma^{(2)}_0(N))$ if and only if $M = \gcd(N_1, N_2)>1$ and the Atkin-Lehner eigenvalues of $f$ and $g$ coincide for all primes $p$ dividing $M$. Moreover, in that case, we have an equality of sets

$$\{F \in S_{\frac{k_1}2+1}(\Gamma^{(2)}_0(N)): F \text{ is a Yoshida lift of }(f,g)\} = \bigcup_{M_1}\{t  F_{f,g;M_1} : t\in \C \}$$ where the union is taken over all positive divisors $M_1$ of $M$  with an odd number of prime factors.
\end{proposition}
\begin{proof}
By Theorem~\ref{t:yoshida}, for a Yoshida lift $F$ of $(f,g)$ to exist,   $\pi_{f,p}$ and $\pi_{g,p}$ must both be discrete series at some prime dividing $M$. So $M$ must be greater than 1 for a lift to exist. Now, let $F$ be a Yoshida lift of $(f,g)$ such that $F \in S_{\frac{k_1}2+1}(\Gamma^{(2)}_0(N))$ where $N = \mathrm{lcm}(N_1, N_2)$. The fact that $N$ is squarefree implies that for each $p$ dividing $N$, $\Phi_F$ maps under the isomorphism $\pi_F \simeq \pi_{F,p} \otimes (\otimes'_{v\neq p} \pi_{F,v})$ to a (non-zero) vector of the form $\phi_p \otimes \alpha$ where $\phi_p$ is a vector in $\pi_{F, p}$ that is fixed by \begin{equation}\label{Gamma0localdefeq}
K_{0,p}^{(2)} := \left\{\begin{pmatrix}A&B\\ C&D \end{pmatrix} \in G_2(\Z_p)\;|\;C \equiv 0 \pmod{p\Z_p}\right\}.
\end{equation}
However, if in addition, we know that $p$ divides $M$ and the Atkin-Lehner eigenvalues of $f$ and $g$ are different at $p$, then the calculations of~\cite[Sec. 3.3]{sahaschmidt} show that a representation $\pi$ of $G_2(\Q_p)$ satisfying $L(\pi_{p}) = L(\pi_{f,p}) \oplus L(\pi_{g,p})$  has no $K_{0,p}^{(2)}$-fixed vector. It follows that the Atkin-Lehner eigenvalues of $f$ and $g$ must coincide at all primes dividing $M$. Conversely, if $M>1$, and the Atkin-Lehner eigenvalues of $f$ and $g$ coincide at all primes dividing $M$, then there does exist a Yoshida lift $F$ of $(f,g)$ with $F \in S_{\frac{k_1}2+1}(\Gamma^{(2)}_0(N))$; indeed any $F_{f,g; M_1}$ as defined above is an example.

Finally, in order to show the equality of the two sets, let $F \in S_{\frac{k_1}2+1}(\Gamma^{(2)}_0(N))$ be a Yoshida lift of $(f,g)$. Let $M_1$ be the product of all the primes $p$ dividing $M$ where $\pi_{F,p}$ is non-generic. By Theorem~\ref{t:roberts}, $M_1$ has an odd number of prime factors. We claim that $F$ is a multiple of $F_{f,g;M_1}$. To show this, we first prove that $\pi_F = \Pi_{f,g;M_1}$. By the multiplicity one assertion of Theorem~\ref{t:roberts}, it suffices to prove that $\pi_F \simeq \Pi_{f,g;M_1}$. We proceed locally. Note that $L((\Pi_{f,g;M_1})_v)  = L(\pi_{F,v})$ for all places $v$. If $v$ is a place that is not among the primes dividing $M$, then there is only one representation in the local $L$-packet, and so it is clear that $\pi_{F,v} \simeq (\Pi_{f,g;M_1})_v.$ And if $p$ divides $M$, then there are two representations in the local $L$-packet $L(\pi_{f,p}) \oplus L(\pi_{g,p})$, one generic and the other non-generic, and by our construction, both $\pi_{F,v}$ and $(\Pi_{f,g,M_1})_v$ pick out the same member, namely the non-generic one if $p$ divides $M_1$ and the generic one if $p$ does not divide $M_1$.

Thus we have shown that $\pi_F = \Pi_{f,g;M_1}$. To complete the proof that $F$ is a multiple of $F_{f,g;M_1}$ it suffices to show that the space of $K_{0,p}^{(2)}$-fixed vectors in $\Pi_{f,g;M_1}$ is one-dimensional for each $p$ dividing $N$. This follows from table $(19)$ of~\cite{sahaschmidt}.
\end{proof}

Next, we define the concept of the Yoshida space attached to a pair of holomorphic newforms. Let $\chi_1$, $\chi_2$ be Dirichlet characters associated to the same primitive Dirichlet character $\chi$ and let $f,g$ be newforms that satisfy the conditions for a scalar valued Yoshida lifting (see Definition~\ref{d:yosh}).

\begin{definition}\label{defyoshspa}For any two classical holomorphic newforms $f$, $g$ satisfying the conditions for a scalar valued Yoshida lifting, the Yoshida space attached to $f, g$, denoted $Y[f,g]$, is the subspace generated by all the possible Yoshida lifts of $(f,g)$.
 \end{definition}

 Equivalently, an element $F$ belongs to $Y[f,g]$ if and only if all irreducible subrepresentations $\pi_F'$ of $\pi_F$ satisfy \begin{equation}\label{lparsum3}L(\pi'_{F,v}) = L(\pi_{f,v}) \oplus L(\pi_{g,v}) \quad \text{ for all places $v$.}\end{equation}
Note that if the weight of $f$ is $k$ and that of $g$ is $2$, and $\chi$ denotes the common primitive character attached to their characters, then  $Y[f,g]$ is a subspace of  $S_{k/2 + 1}^{(2)}(\chi).$

\subsection{Weak Yoshida lifts and weak Yoshida spaces}\label{s:weakyosh}
Let $f$, $g$ be newforms that satisfy the conditions for a scalar valued Yoshida lifting (see Definition~\ref{d:yosh}). Without loss of generality, we may assume that the weight of $f$ is $k$ and that of $g$ is $2$.

 \begin{definition}\label{defweakyosh}Given $f$, $g$ as above, a non-zero element $F\in S_{k/2 + 1}^{(2)}$ is said to be a \emph{weak Yoshida lift} of $(f, g)$ if all the irreducible subrepresentations of $\pi_F$ are isomorphic to a common representation $\pi_F'$  and  \begin{equation}\label{lparsum2}L(\pi'_{F,p}) = L(\pi_{f,p}) \oplus L(\pi_{g,p}) \quad \text{ for \emph{almost} all primes $p$.}\end{equation}
 \end{definition}
  Equivalently, a non-zero Siegel cusp form $F$ is a weak Yoshida lift of $(f,g)$ if and only if there exists an irreducible cuspidal representation $\pi$ of $G_2(\A)$ such that  \begin{enumerate}
\item $L(\pi_{p}) = L(\pi_{f,p}) \oplus L(\pi_{g,p})$  for almost all primes $p$.\footnote{This condition uniquely characterizes $\pi_p$ at all $p$ where $\pi_{f,p}$ and $\pi_{g,p}$ are both unramified principal series.}
\item $F \in S_{k/2 + 1}(\Gamma^{(2)}(N) ; \pi_\mathfrak{f})$ for some $N$.

\end{enumerate}

 Note that it is part of the definition of a weak Yoshida lift that it's weight must be $\frac{k}2 + 1$ (where $k$ is the larger of the weights of $f$ and $g$). This may appear restrictive but it is actually not. Indeed, using the global functional  equation for the $L$-function of $\pi'_F$, one can show that any $F$ satisfying~\eqref{lparsum2} must have weight $k/2 + 1$. We suppress further discussion of this point in the interest of brevity.

 Any weak Yoshida lift of $(f,g)$ is automatically an element of $S_{k/2+1}^{(2)}(\chi)$ (where $\chi$ is the common Dirichlet character associated to the characters of $f$ and $g$) and is an eigenfunction of $\mathcal{H}_{p,1}^{\mathrm{class}, (2)}$ for almost all $p$. Note also that any Yoshida lift is a weak Yoshida lift. In Section~\ref{s:multiplicity}, we will show conditionally (assuming some special cases of Langlands functoriality) that any weak Yoshida lift is a Yoshida lift.

 Finally, we define the concept of the \emph{weak Yoshida space} attached to a pair of classical newforms.

 \begin{lemma}\label{weakyoshdeflemma}Let $f$, $g$ be classical holomorphic newforms  of weights $k$ and 2 respectively, satisfying the conditions for a scalar valued Yoshida lifting (see Definition~\ref{d:yosh}). Let $\chi$ be the primitive Dirichlet character attached to the characters of $f$ and $g$, and let $F \in S_{k/2 + 1}^{(2)}$. The following are equivalent.
  \begin{enumerate}

  \item  $F = \sum_{i=1}^r F_i$ with each $F_i$ a weak Yoshida lift of $(f,g)$.

  \item   If $N$ is any integer such that  $F \in S_{k/2 + 1}(\Gamma^{(2)}(N) )$, then we have an expression  $F= \sum_{i=1}^r F_i$ with each $F_i$ a weak Yoshida lift of $(f,g)$ and each $F_i \in S_{k/2 + 1}(\Gamma^{(2)}(N) )$.

   \item  All irreducible subrepresentations $\pi_F'$ of $\pi_F$ satisfy $$L(\pi'_{F,p}) = L(\pi_{f,p}) \oplus L(\pi_{g,p}) \quad \text{ for almost all primes $p$.}$$

   \item  For each $p$ not dividing $N_1$, $N_2$, let $\Pi_p$ denote the (unique up to isomorphism) unramified principal series representation of $G_2(\Q_p)$ satisfying $$L(s, \Pi_p) = L(s, \pi_{f,p}) L(s, \pi_{g,p}).$$  Then there exists a set $S$ of primes containing almost all primes, and an integer $N$, such that   $F \in S_{k/2 + 1}(\Gamma^{(2)}(N) ; \Pi_S)$, where $\Pi_S = \otimes_{p \in S} \Pi_p$.

  \item  For almost all primes  $p$ where $F$ is $p$-spherical,  $F$ is an eigenfunction for the Hecke operators $T^{(2)}(p)$, $T_1^{(2)}(p^2)$ with the eigenvalues determined by those of $f$ and $g$ as follows, $$ \lambda^{(2)}(F,p) = p(\lambda^{(1)}(f,p) + \lambda^{(1)}(g,p)), \qquad  \lambda_1^{(2)}(F, p^2) = \chi(p)(p^2-1) + p\lambda^{(1)}(f,p)\lambda^{(1)}(g,p).$$

\end{enumerate}

\noindent Moreover, for each $f$, $g$ as above, the set of $F \in S_{k/2 + 1}^{(2)}$ satisfying any (and hence all) of the above conditions is a vector subspace of   $S_{k/2 + 1}^{(2)}(\chi)$.
\end{lemma}

\begin{proof}The equivalence of the first four statements follow directly from the relevant definitions and Theorem~\ref{t:deadelize}. The fifth statement follows from the same argument as in Remark~\ref{yoshbasicrem2}. The assertion about being a subspace is trivial.
\end{proof}

\begin{definition}\label{defweakyoshspa}For any two classical holomorphic newforms $f$, $g$ satisfying the conditions for a scalar valued Yoshida lifting, we define the weak Yoshida space attached to $f, g$, denoted $Y'[f,g]$, to consist of all forms $F$ satisfying any (and hence all) of the conditions of Lemma~\ref{weakyoshdeflemma}.
\end{definition}

The really nice thing about the space  $Y'[f,g]$ is that it can be defined --- by the last equivalent assertion in Lemma~\ref{weakyoshdeflemma} --- using completely classical language that requires no notion of adelization or cuspidal representations or $L$-parameters. Note that given any pair of newforms $f$, $g$ satisfying the conditions for a scalar valued Yoshida lifting, we have an inclusion of sets $$\emptyset \neq \text{Yoshida lifts} \subseteq \text{weak Yoshida lifts} \subseteq \text{weak Yoshida space}.$$

\begin{remark} Assuming that there exists a strong functorial lifting from $\GSp_4$ to $\GL_4$, we will prove in Section~\ref{s:multiplicity} that every weak Yoshida lift is in fact a Yoshida lift, and hence the weak Yoshida space $Y'[f,g]$ coincides with the Yoshida space $Y[f,g]$.
\end{remark}

\subsection{Arithmeticity for Yoshida liftings}
The point of this short subsection is to provide a theorem which gathers together some arithmeticity properties for (weak) Yoshida lifts and (weak) Yoshida spaces.

We begin with a lemma.

\begin{lemma}\label{lemmaarithmeticityyosh}Let $p$ be a prime and $\pi_1$, $\pi_2$ be two representations of $G_1(\Q_p)$ with the same central character. Suppose $\Pi$ is a representation of $\grave{}G_2(\Q_p)$ such that $L(\Pi) = L(\pi_1) \oplus L(\pi_2)$. Then \begin{enumerate}
\item For any $\sigma \in \Aut(\C)$, $L(\leftexp{\sigma}{\Pi}) = L(\leftexp{\sigma}{\pi_1}) \oplus L(\leftexp{\sigma}{\pi_2})$.

\item $\Q(\Pi) \subset \Q(\pi_1)\Q(\pi_2)$.
\end{enumerate}
\end{lemma}
\begin{proof}We begin an easy observation.  Suppose $\chi$ is a character of $\Q_p^\times$. We can think of $\chi$ as a character of $G_n(\Q_p)$ by composing it with the similitude character $\mu_n$. Then $L(\Pi) = L(\pi_1) \oplus L(\pi_2)$ if and only if $L(\Pi \otimes \chi) = L(\pi_1 \otimes \chi ) \oplus L(\pi_2 \otimes \chi)$. The proof of this observation follows immediately from the behavior of $L$-parameters under twisting~\cite[Main Theorem, part (iv)]{gantakGSp4}.

Now, let $\Pi$, $\pi_1$, $\pi_2$ be as in the statement of the lemma and $\sigma \in \Aut(\C)$. Then, using the above observation, we have  $L(\Pi \otimes ||_p^{1/2}) = L(\pi_1 \otimes ||_p^{1/2}) \oplus L(\pi_2 \otimes ||_p^{1/2}).$ Applying $\sigma$, we deduce $\sigma(L(\Pi \otimes ||_p^{1/2})) = \sigma(L(\pi_1 \otimes ||_p^{1/2})) \oplus \sigma(L(\pi_2 \otimes ||_p^{1/2})).$ Using~\cite[Lemma 5.3]{morimoto}, we see that this implies that $L(\leftexp{\sigma}{\Pi} \otimes \leftexp{\sigma}{||_p^{1/2}}) = L(\leftexp{\sigma}{\pi_1} \otimes \leftexp{\sigma}{||_p^{1/2}}) \oplus L(\leftexp{\sigma}{\pi_2} \otimes \leftexp{\sigma}{||_p^{1/2}}).$ Using the observation again, we deduce that $L(\leftexp{\sigma}{\Pi}) = L(\leftexp{\sigma}{\pi_1}) \oplus L(\leftexp{\sigma}{\pi_2})$.

To prove the second assertion of the lemma, let $\sigma \in \Aut(\C)$ be such that $\leftexp{\sigma}{\pi_1} = \pi_1$ and $\leftexp{\sigma}{\pi_2} = \pi_2$. We need to show that $\leftexp{\sigma}{\Pi} = \Pi$. Using the first part proved above, it follows that $L(\leftexp{\sigma}{\Pi}) = L(\Pi).$ It is known~\cite{rob2001, gantakGSp4} that there are either one or two representations in any given $L$-packet and (when there are two of them) exactly one is generic. The key point now is that $\leftexp{\sigma}{\Pi}$ is generic if and only if $\Pi$ is; this follows immediately from the easy fact that if $l$ is a non-zero Whittaker functional on $\Pi$ then $\sigma \circ l$ is a   non-zero Whittaker functional on $\leftexp{\sigma}{\Pi}$. As a result, $L(\leftexp{\sigma}{\Pi}) = L(\Pi)$ implies that $\leftexp{\sigma}{\Pi} = \Pi$, as desired.
\end{proof}

\begin{theorem}\label{rationalyoshida}Let $ f $ and $g$  be classical holomorphic newforms which satisfy the conditions for a scalar valued Yoshida lifting. Let $\pi_f$ and $\pi_g$ be the irreducible cuspidal representations of $\GL_2(\A)$ attached to $f$ and $g$ respectively, and let $\Q(f,g)$ denote the field $\Q((\pi_f)_\mathfrak{f}) \Q((\pi_g)_\mathfrak{f}).$ Let $N$ be any positive integer.
  \begin{enumerate}
\item \textbf{(Yoshida lift)}  Suppose that $F \in S_k(\Gamma^{(2)}(N))$ is a Yoshida lift of $(f,g)$. Then, $ \Q((\pi_F)_\mathfrak{f}) \subset \Q(f,g)$ and for any $\sigma \in \Aut(\C)$, we have that $\leftexp{\sigma}{F}$ is a Yoshida lift of $(\leftexp{\sigma}{f},\leftexp{\sigma}{g})$. Furthermore, we can write $F$ as a finite sum $F =\sum_i c_i F_i$,  such that \begin{enumerate}

        \item $F_i$ is a Yoshida lift of $(f, g)$ and $\pi_{F_i} = \pi_F$.

    \item $F_i \in  S_k(\Gamma^{(2)}(N); \Q(f,g) )$; in particular all the Fourier coefficients of $F_i$ are contained in a $\CM$ field.

    \item $c_i \in L \Q(f,g)$ where $L$ is any subfield of $\C$ containing all the Fourier coefficients of $F$.

     \item $\langle F_i, F_j \rangle =0 $ for $i \neq j$.

    \end{enumerate}

\item  \textbf{(weak Yoshida lift) }Suppose that $F \in S_k(\Gamma^{(2)}(N))$ is a weak Yoshida lift of $(f,g)$. Then, for any $\sigma \in \Aut(\C)$, $\leftexp{\sigma}{F}$ is a weak Yoshida lift of $(\leftexp{\sigma}{f},\leftexp{\sigma}{g})$. Moreover, if $\pi'_F$  is any irreducible subrepresentation of $\pi_F$ (note that all irreducible subrepresentations of $\pi_F$ are isomorphic), then we can write $F$ as a finite sum $F =\sum_i c_i F_i$,  such that \begin{enumerate}
     \item $F_i$ is a weak Yoshida lift of $(f, g)$ and $\pi'_{F_i} \simeq \pi'_F$ for each irreducible subrepresentation $\pi'_{F_i}$ of $\pi'_F$.
    \item $F_i \in  S_k(\Gamma^{(2)}(N); \Q((\pi'_F)_\mathfrak{f}) )$; in particular all the Fourier coefficients of $F_i$ are contained in some $\CM$ field.

    \item $c_i \in L\Q((\pi'_F)_\mathfrak{f})$ where $L$ is any subfield of $\C$ containing all the Fourier coefficients of $F$.

     \item $\langle F_i, F_j \rangle =0 $ for $i \neq j$.
\end{enumerate}

\item \textbf{(weak Yoshida space)} Suppose that $F \in S_k(\Gamma^{(2)}(N))$ belongs to the weak Yoshida space $Y'[f,g]$. Then, for any $\sigma \in \Aut(\C)$, $\leftexp{\sigma}{F}$ belongs to the weak Yoshida space $Y'[\leftexp{\sigma}{f},\leftexp{\sigma}{g}]$. Moreover, we can write $F$ as a finite sum $F =\sum_i c_i F_i$,  such that

    \begin{enumerate}
     \item $F_i \in Y'[f,g]$.
    \item $F_i \in  S_k(\Gamma^{(2)}(N); \Q(f, g))$; in particular all the Fourier coefficients of $F_i$ are contained in a $\CM$ field.

    \item $c_i \in L\Q(f,g)$ where $L$ is any subfield of $\C$ containing all the Fourier coefficients of $F$.

     \item $\langle F_i, F_j \rangle =0 $ for $i \neq j$.

    \end{enumerate}

   \item \textbf{(ratio of Petersson norms)} Suppose that $F$ and $G$ are non-zero elements belonging to the weak Yoshida space $Y'[f,g]$ and that all the Fourier coefficients of $F$ and $G$ lie in $\Q^{\CM}$. Then $\frac{\langle G, G\rangle}{\langle F,F \rangle}$ lies in $\Q^{\CM}$, and moreover, for all $\sigma \in \Aut(\C)$, we have $$ \sigma \left( \frac{\langle G, G\rangle}{\langle F,F \rangle}  \right) =  \frac{\langle \leftexp{\sigma}G, \leftexp{\sigma}G\rangle}{\langle \leftexp{\sigma}F,\leftexp{\sigma}F \rangle}.$$

\end{enumerate}
\end{theorem}
\begin{proof}
The assertions about the Yoshida lift and the weak Yoshida lift follow from Lemma~\ref{lemmaarithmeticityyosh} and Theorem~\ref{t:aithmeticity} by taking $S = \mathfrak{f}$. Note here that in the case of Yoshida lifts, the assertion $F_i \in S_k(\Gamma^{(2)}(N); (\pi_F)_{\mathfrak{f}} )$ implies that $\pi_{F_i} = \pi_F$ because of multiplicity one.

The assertion about the weak Yoshida space follows also  from Lemma~\ref{lemmaarithmeticityyosh} and Theorem~\ref{t:aithmeticity} by taking $S$ to be a set of primes (containing almost all primes) such that for all primes $p$ in $S$ each of $F$, $f$, $g$ are $p$-spherical and eigenfunctions of the local Hecke algebra at $p$. Note here that for any such prime $p$ we have $\Q((\pi_F')_p) \subset \Q(\pi_{f,p})\Q(\pi_{g,p})$; this is immediate by looking at the $L$-function.

Finally, the assertion about Petersson norms is an immediate consequence of Corollary~\ref{corpeteq}.
\end{proof}

\subsection{Do there exist weak Yoshida lifts that are not Yoshida lifts?}\label{s:multiplicity}

It is natural to wonder whether there exists a weak Yoshida lift that is not a Yoshida lift. In this subsection, we show that the answer is \emph{no} provided we assume certain well-known conjectures that are expected to be true (but not completely proved at the time of this writing). The results of this subsection will not be used elsewhere in this paper.

\begin{definition}Given a positive integer $k$ and a congruence subgroup $\Gamma$ of $\Sp_4(\Q)$, we say that there exists a nice $L$-function theory for $S_k(\Gamma)$  if for all $F \in S_k(\Gamma)$ and each irreducible subrepresentation $\pi$ of $\pi_F$, the global $L$-functions $L(s, \pi)$, $L(s, \hat{\pi})$   have meromorphic continuation to the entire complex plane and satisfy the functional equation $L(s, \pi) = \varepsilon(s, \pi) L(1-s, \hat{\pi})$. Here $L(s, \pi)= \prod_v L(s, \pi_v)$, $L(s, \hat{\pi})= \prod_v L(s, \hat{\pi_v})$, $\varepsilon(s, \pi) = \prod_v\varepsilon(s, \pi_v)$ are  Euler products of local factors defined via the Local Langlands correspondence.

\end{definition}

The existence of a nice $L$-function theory for $S_k(\Sp_4(\Z))$ was proved by Andrianov~\cite{andrianov}. In general, the existence of a nice $L$-function theory for all $S_k(\Gamma)$ (and indeed, for all irreducible cuspidal representations of $G_2(\A)$) follows \emph{in principle} from the work of Piatetski-Shapiro~\cite{psg}. Unfortunately, the explicit verification that the local $L$ and $\varepsilon$-factors defined by Piatetski-Shapiro coincides with those defined by the local Langlands correspondence does not appear to have been carried out except for certain special types of congruence subgroups. For the application below, we only need to assume the existence of a  nice $L$-function theory for certain $\Gamma$ with squarefree level.

 For any integer $M$, we let $B(M)$ denote the Borel-type congruence subgroup of $\Sp_4(\Z)$ of level $M$, namely, $$B(M) = Sp(4,\Z) \cap \begin{pmatrix}\Z& M\Z&\Z&\Z\\\Z& \Z&\Z&\Z\\M\Z& M\Z&\Z&\Z\\M\Z&M \Z&M\Z&\Z\\\end{pmatrix}.$$

\begin{proposition}\label{p:strongmult}Let $\Gamma$ be a congruence subgroup of $\Sp_4(\Z)$ containing $B(N)$  for some squarefree integer $N$. Assume that there exists a nice $L$-function theory for $S_k(\Gamma)$. Let $F$, $G$ be elements of $S_k(\Gamma)$. Let $\pi_1$ be an irreducible subrepresentation of $\pi_F$ and $\pi_2$ be an irreducible subrepresentation of $\pi_G$ such that $\pi_{1,p}$, $\pi_{2,p}$ are tempered at all primes $p$ and $\pi_{1, p} \simeq \pi_{2,p}$ for almost all primes $p$. Then $L(\pi_{1,p}) = L(\pi_{2,p})$ for all primes $p$.
\end{proposition}
\begin{proof}
Note that $\pi_{1,v}$ and $\pi_{2,v}$ are isomorphic at $v = \infty$ and at almost all finite primes $v=p$. So, dividing the functional equation of $L(s, \pi_1)$ by the functional equation of $L(s,\pi_2)$ (this is where we use our assumption that there exists a nice $L$-function theory for $S_k(\Gamma)$), and using~\cite[Lemma 3.1.1]{sch} we get, for each prime $p$, an equation $$\frac{L(s, \pi_{1,p})}{L(s, \pi_{2,p})} = cp^{ms} \frac{L(1-s, \widehat{\pi_{1,p}})}{L(1-s, \widehat{\pi_{2,p}})}$$ where $m$ is some integer and $c$ is some complex number (both depending on $p$). Since $\Gamma$ contains $B(N)$, it follows that both $\pi_{1,p}$ and $\pi_{2,p}$ are Iwahori spherical representations; furthermore they are both tempered. Now, going through the various possibilities listed in table 2 of~\cite{sch}, we can check that for any pair $\pi_{1,p}$, $\pi_{2,p}$ satisfying all the above conditions we have that $L(\pi_{1,p}) = L(\pi_{2,p})$.\footnote{I thank Ralf Schmidt for clarifying this point to me.}
\end{proof}

\begin{remark}Proposition~\ref{p:strongmult} is false without the temperedness assumption. Indeed, the various Saito-Kurokawa lifts of square-free level described in~\cite{schsaito} give rise to nearly equivalent representations of $G_2(\A)$ whose local components are isomorphic at almost all places yet do not even belong to the same $L$-packet in the remaining places.
\end{remark}

\begin{remark}If $\Gamma = \Sp_4(\Z)$, and $\pi_1$, $\pi_2$ are as in Proposition~\ref{p:strongmult}, the same proof shows in fact that $\pi_{1,p} \simeq \pi_{2,p}$ at all primes $p$. For general $\Gamma$, this is not true; indeed the various Yoshida lifts of some pair $(f,g)$ give rise to nearly equivalent but non-isomorphic representations.

\end{remark}

\begin{corollary}\label{weakyoshcor}Let $f$, $g$ be classical holomorphic newforms which satisfy the conditions for a scalar valued Yoshida lifting. Let $\Gamma$ be a congruence subgroup of $\Sp_4(\Z)$ such that $F \in S_k(\Gamma)$ is a Yoshida lift of $(f,g)$ and $G \in S_k(\Gamma)$ is a weak Yoshida lift of $(f,g)$. Assume that there exists a nice $L$-function theory for $S_k(\Gamma)$ and that there exists a squarefree integer $N$ such that $\Gamma$ contains $B(N)$. Then $G$ is a Yoshida lift of $(f,g)$.
\end{corollary}
\begin{proof}By looking at the primes $p$ where~\eqref{lparsum2} is satisfied by both $F$ and $G$, we conclude that neither is of  Saito-Kurokawa type. A result of Weissauer~\cite{weissram} now tells us that the irreducible subrepresentations of $\pi_G$ and $\pi_F$ are tempered everywhere. The corollary follows from Proposition~\ref{p:strongmult} and the multiplicity one property mentioned earlier (Theorem~\ref{t:roberts}).
\end{proof}

Corollary~\ref{weakyoshcor} tells us that weak Yoshida lifts are exactly the same as Yoshida lifts provided all the relevant Siegel cusp forms are defined with respect to a Borel-type congruence subgroup of squarefree level and we assume the existence of a nice $L$-function theory for such forms. We will now prove that weak Yoshida lifts and Yoshida lifts are \emph{always} the same if we make a stronger assumption.

Let $v$ be a place of $\Q$. Let $\Phi(\GL_4(\Q_v))$ denote the set of isomorphism classes of admissible homomorphism $\phi:W_v' \rightarrow \GL_4(\C)$ where $W_v'$ denotes the Weil-Deligne group of $\Q_v$. Let $\Pi(\GL_4(\Q_v))$ denote the set of isomorphism classes of irreducible admissible representations of $\GL_4(\Q_v)$. By local Langlands for $\GL_n$, we have a bijective map
\begin{equation}\label{locallanglands2}L:\Pi(\GL_4(\Q_v)) \rightarrow \Phi(GL_4(\Q_v)),\end{equation} satisfying certain conditions (including, in particular, compatibility of $L$ and $\varepsilon$-factors).

\begin{definition}Given a positive integer $k$ and a congruence subgroup $\Gamma$ of $\Sp_4(\Q)$, we say that there exists a strong functorial lifting to $\GL_4$ for $S_k(\Gamma)$  if for all $F \in S_k(\Gamma)$ and each irreducible subrepresentation $\pi = \otimes_v \pi_v$ of $\pi_F$, there exists an irreducible automorphic representation\footnote{For the definition of an automorphic representation, we refer the reader to~\cite{borel-jacquet}.} $\pi'=\otimes_v\pi'_v$ of $\GL_4(\A)$ such that $L(\pi_v) = L(\pi'_v)$ for all places $v$.

\end{definition}

The existence of a strong functorial lifting to $\GL_4$ for $S_k(\Gamma)$ for all $k$ and $\Gamma$ is predicted by Langlands' general conjectures. A proof should follow from the far-reaching work of Arthur on classification of automorphic representations using the trace formula. At the time of this writing, it appears that Arthur's results are still conditional on certain stabilization results. We note that the existence of such a lifting for $S_k(\Gamma)$ certainly implies the existence of a nice $L$-function theory for $S_k(\Gamma)$. However, proving functorial lifting  is generally a much deeper problem than proving niceness of a single sort of $L$-function. If $\Gamma = \Sp_4(\Z)$, a strong functorial lifting to $\GL_4$ for $S_k(\Gamma)$ was proved in joint work of the author with Pitale and Schmidt~\cite{transfer} (and it was also proved that this lifting is cuspidal provided $F$ is not  a Saito-Kurokawa lift).

\begin{proposition}Let $f$, $g$ be classical holomorphic newforms that satisfy the conditions for a scalar valued Yoshida lifting. Let $\Gamma$ be a congruence subgroup of $\Sp_4(\Z)$ such that $F \in S_k(\Gamma)$ is a weak Yoshida lift of $(f,g)$. Assume that there exists a strong functorial lifting to $\GL_4$ for $S_k(\Gamma)$. Then $F$ is a Yoshida lift of $(f,g)$.

\end{proposition}
\begin{proof}As before, we deduce that $F$ is not of Saito-Kurokawa type. Let $\pi$ be any irreducible subrepresentation of $\pi_F$ and $\pi'$ be an irreducible automorphic representation of $\GL_4(\A)$ such that $L(\pi_v) = L(\pi'_v)$ for all places $v$. Thus $L(\pi'_v) = L(\pi_{f, v}) \oplus L(\pi_{g, v})$ for almost all places $v$. By the result of Weissauer~\cite{weissram}, $\pi'$ is tempered everywhere. We need to prove that $L(\pi'_v) = L(\pi_{f, v}) \oplus L(\pi_{g, v})$ at the remaining places $v$. By Theorem 4.4 of~\cite{jacquet-shalika-81-2}, we know that \emph{every} place $v$, $\pi'_v$ is a constituent of $\mathrm{Ind}_{P(\Q_v)}^{\GL_4(\Q_v)}(\pi_{f, v} \otimes \pi_{g, v})$, where $P$ denotes the $(2,2)$ parabolic of $\GL_4$. Now, the results of Bernstein and Zelevinsky~\cite{berzel}, along with the fact that the representations $\pi_{f, v}$, $\pi_{g, v}$, $\pi'_{v}$ are all tempered, identify $\pi'_{v}$ uniquely, and imply in particular that $L(\pi'_v) = L(\pi_{f, v}) \oplus L(\pi_{g, v})$, as required.

\end{proof}

\section{Ratio of Petersson norms and special $L$-values}

\subsection{The main result}\label{s:mainres} Let us recall some of the concepts defined in the previous section. Given two classical holomorphic newforms $f$ and $g$ that are not multiples of each other and whose nebentypes are associated to the same primitive character, we defined the notion of a \emph{Yoshida lift} (Definition~\ref{defyosh}). We showed that a Yoshida lift exists if and only if $f$ and $g$ \emph{satisfy the conditions for a scalar valued Yoshida lifting}, a phrase that was precisely defined in Definition~\ref{d:yosh}. Assuming that $f$ and $g$ are newforms satisfying these conditions, we also defined the \emph{Yoshida space} $Y[f,g]$ (Definition~\ref{defyoshspa}), the \emph{weak Yoshida lifts} (Definition~\ref{defweakyosh}), and the \emph{weak Yoshida space} $Y'[f,g]$ (Definition~\ref{defweakyoshspa}). All these lifts/spaces consisted of Siegel cusp forms of degree 2 and weight $k/2 + 1$ (assuming the weight of $g$ is 2 and that of $f$ is $k$). We noted that the vector space $Y'[f,g]$ contains the span of all the Yoshida lifts; we also showed that it is in fact exactly equal to this span if we assume the existence of a Langlands transfer from $\GSp_4$ to $\GL_4$.

We now state the main result of this paper.

\begin{theorem}\label{t:main} Let $f $, $g$ be classical holomorphic newforms, of weights $2k$ and 2 respectively, that satisfy the conditions for a scalar valued Yoshida lifting. Let $F \in Y'[f,g]$. Suppose that $k \ge 6$ and that all the Fourier coefficients of $f, g$ and $F$ belong to $\Q^{\CM}$. Then $$\frac{\pi^2\langle F, F\rangle}{\langle f, f\rangle} \in \Q^{\CM},$$ and moreover, for $\sigma \in \Aut(\C)$, we have $$\sigma\left(\frac{\pi^2\langle F, F\rangle}{\langle f, f\rangle}\right) =  \frac{\pi^2\langle {}^\sigma\!F, {}^\sigma\!F\rangle}{\langle {}^\sigma\!f, {}^\sigma\!f\rangle}.$$
\end{theorem}

\begin{remark}By Theorem~\ref{rationalyoshida}, there is no real loss of generality in assuming that all the forms involved have $\CM$ algebraic Fourier coefficients.
\end{remark}

\begin{remark}Let $f $, $g_1$, $g_2$ be classical holomorphic newforms such that the pairs $(f,g_1)$ and $(f, g_2)$ both satisfy the conditions for a scalar valued Yoshida lifting (where $g_1$ and $g_2$ both have weight 2). Suppose that $F_1 \in Y'[f, g_1]$, and $F_2 \in Y'[f, g_2]$. Then, as a corollary to Theorem~\ref{t:main}, we deduce (under the standard algebraicity assumptions on the Fourier coefficients) that the ratio $\frac{\langle F_1, F_1\rangle}{\langle F_2, F_2\rangle}$ is algebraic and satisfies the usual  $\Aut(\C)$ equivariance. This is quite surprising because $F_1$ and $F_2$ have different Hecke eigenvalues at infinitely many primes provided $g_1$ and $g_2$ are not multiples of each other, and so the algebraicity of the  ratio $\frac{\langle F_1, F_1\rangle}{\langle F_2, F_2\rangle}$ goes far beyond what Corollary~\ref{corpeteq} provides.
\end{remark}

\medskip

As an initial key step to proving Theorem~\ref{t:main}, we observe that it suffices to prove it in the case that $F$ is actually a Yoshida lift. This is because of Corollary~\ref{corpeteq}, and the fact that given $f$, $g$ as in Theorem~\ref{t:main}, there does indeed exist a Yoshida lift with $\CM$ algebraic Fourier coefficients.

\subsection{Special value results for certain $L$-functions}

In this subsection, we will recall certain special $L$-values results that will be crucial for the proof of Theorem~\ref{t:main}. Let us first describe the general philosophy. If $L(s, \mathcal{M})$ is an arithmetically defined (or motivic) $L$-series associated to an arithmetic object $\mathcal{M}$, it is of interest to study its values at certain critical points $s=m$. For these critical points, conjectures due to Deligne predict that  $L(m,\mathcal{M})$ is the product of a suitable transcendental number $\Omega$ and an algebraic number $A(m,\mathcal{M})$, and furthermore, if $\sigma$ is an automorphism of $\C$, then $\sigma(A(m,\mathcal{M}))= A(m, \leftexp{\sigma}{\mathcal{M}})$.
In this subsection, we will recall some critical value results in the spirit of the above conjecture for $L$-functions associated to $\GL_2 \times \GL_2$ and $\GSp_4 \times \GL_2$. We will recall the results only in the form that we need (and not strive to state the most general result that is known).
\subsubsection*{Special values for $\GL_2 \times \GL_2$}
Let $g \in S_{k}(M, \chi)$, $h \in S_l( N, \psi)$ be classical holomorphic newforms. Then, $\pi_g$ and $\pi_h$ are irreducible, cuspidal representations of $\GL_2(\A)$. The $L$-function $L(s, \pi_f \times \pi_g) = \prod_v L(s, \pi_{f,v} \times \pi_{g,v}) $ is defined using the local Langlands correspondence (it is known to coincide with the definition obtained from the Rankin-Selberg method, or the method of Shahidi). We let $L_\mathfrak{f}(s, \pi_f \times \pi_g) = \prod_p L(s, \pi_{f,p} \times \pi_{g,p})$ denote the finite part of this $L$-function. The following result follows directly from Theorem 3 of~\cite{shi76} after adjusting for normalizations.

\begin{theorem}[Shimura~\cite{shi76}]\label{tgl2gl2}Let $g \in S_{k}(M, \chi)$, $h \in S_l( N, \psi)$ be classical holomorphic newforms such that $k>l$ and all the Fourier coefficients of $f$ and $g$ lie in $\Q^{\CM}$. For $m \in \frac12 \Z$, define $$C(m, g, h) = \frac{L_\mathfrak{f}(m, \pi_g\times \pi_h)}{\pi^{2m+k}i^{1-k}G(\chi \psi)\langle g, g \rangle}.$$ Then, for $m \in \left[\frac12, \frac{k-l}2\right] \cap (\Z + \frac{k+l}2)$, we have $C(m, g, h) \in \overline{\Q}$ and for all $\sigma \in \Aut(\C)$, we have $\sigma(C(m, g, h))= C(m, {}^\sigma\!g, {}^\sigma\!h).$
\end{theorem}

\subsubsection*{Special values for $\GSp_4 \times \GL_2$}

Let $F \in S_k^{(2)}$ be such that $\pi_F$ is irreducible. Let $g \in S_{l}(M, \chi)$ be a classical holomorphic newform. The $L$-function $L(s, \pi_F \times \pi_g) = \prod_v L(s, \pi_{F,v} \times \pi_{g,v}) $ is defined using the local Langlands correspondence. We let $L_\mathfrak{f}(s, \pi_F \times \pi_g) = \prod_p L(s, \pi_{F,p} \times \pi_{g,p})$ denote the finite part of this $L$-function. There is a long sequence of results by various people on the critical values of $L_\mathfrak{f}(s, \pi_F \times \pi_g)$. We mention in particular the papers of   Furusawa~\cite{fur}, B\"ocherer--Heim~\cite{heimboch}, and various combinations of Pitale, Schmidt and the author~\cite{pitsch, pitale-bessel, saha1, sah2, transfer}.  However, the most general result in the \emph{equal weight case} has been obtained recently by Morimoto~\cite{morimoto}. The following result follows directly from Theorem I of~\cite{morimoto} after adjusting for normalizations.
\begin{theorem}[Morimoto~\cite{morimoto}]\label{tgsp4gl2}Let $F \in S_k^{(2)}$ be such that $\pi_F$ is irreducible and $g \in S_{k}(M, \chi)$ be a classical holomorphic newform. Suppose that $k>6$ and all the Fourier coefficients of $F$ and $g$ lie in $\Q^{\CM}$. Let $\psi$ be the primitive Dirichlet character canonically attached to the central character of $\pi_F$.\footnote{Equivalently, $F \in S_k^{(2)}(\psi)$.} For $m \in \frac12\Z$, define $$C(m, F, g) = \frac{L_\mathfrak{f}(m, \pi_F\times \pi_g)}{\pi^{4m+3k-1}i^{k}G(\chi \psi)^2\langle F, F \rangle \langle g, g \rangle}.$$ Then, for $m \in \left[3, \frac{k}2 - 1\right] \cap (\Z + \frac{k}2)$, we have $C(m, F, g) \in \overline{\Q}$ and for all $\sigma \in \Aut(\C)$, we have $\sigma(C(m, F, g))= C(m, {}^\sigma\!F, {}^\sigma\!g).$
\end{theorem}

\begin{remark}In his paper~\cite{morimoto}, Morimoto states the above theorem incorrectly as being true whenever the Fourier coefficients of $F$ are algebraic. This cannot be true, as can be seen easily by multiplying $F$ by some algebraic number $\alpha$ that is not in $\Q^{\CM}$ and choosing $\sigma$ such that $\sigma(\overline{\alpha}) \neq \overline{\sigma(\alpha)}$. The error in~\cite{morimoto} can be traced back to his Lemma 6.4, which is stated for any number field $L$ but in reality is true only for $L$ that are subfields of $\CM$ fields (compare with our Corollary~\ref{corpeteq}).
\end{remark}

\begin{remark}Using Corollary~\ref{corpeteq}, it is clear that Theorem~\ref{tgsp4gl2} can be extended to any $F$ in $S_k^{(2)}$ that is an eigenfunction of almost all Hecke operators (and replace $\pi_F$ by some irreducible subrepresentation).

\end{remark}

\subsection{The proof of Theorem~\ref{t:main}}

We now prove Theorem~\ref{t:main}. Let $f \in S_{2k}(N_1, \chi_1)$, $g \in S_{2}(N_2, \chi_2)$ be newforms (normalized to have $\CM$ algebraic Fourier coefficients) that satisfy the conditions for a scalar valued Yoshida lifting. We let $\chi$ denote the primitive Dirichlet character that is associated with both $\chi_1$ and $\chi_2$. Note that $\chi$ is necessarily even, i.e., $\chi(-1) = 1$. As remarked at the end of Section~\ref{s:mainres}, it suffices to prove the Theorem in the special case that $F$ is a Yoshida lift and all the Fourier coefficients of $F$ lie in $\Q^{\CM}$. It is clear that $F \in S_{k+1}^{(2)}(\chi)$. As in the statement of the Theorem, we assume that $k \ge 6$.

We have two cases depending on whether $k$ is odd or even. We first deal with the case $k$ odd. Then $k+1$ is even, and since $\overline{\chi}$ is an even Dirichlet character, for some sufficiently large level there exists\footnote{This follows easily from weak dimension formulas for spaces of cusp forms.} a classical holomorphic newform $h$ of weight $k+1$ and character $\overline{\chi}$. By normalizing, we may assume that $h$ has $\CM$ algebraic Fourier coefficients.

Define $$C(F, h)= \frac{L_\mathfrak{f}((k-1)/2, \pi_F\times \pi_h)}{\pi^{5k-3}\langle F, F \rangle \langle h, h \rangle}.$$

Then, by Theorem~\ref{tgsp4gl2}, for any $\sigma \in \Aut(\C)$ we have \begin{equation}\label{proof1}\sigma(C(F, h)) = C({}^\sigma\!F, {}^\sigma\!h).
\end{equation}

Next, define $$C(f, h)= \frac{L_\mathfrak{f}((k-1)/2, \pi_f\times \pi_h)}{\pi^{3k-1}i \langle f, f \rangle}, \quad C(h, g)= \frac{L_\mathfrak{f}((k-1)/2, \pi_h\times \pi_g)}{\pi^{2k}i \langle h, h \rangle}$$

Then, by Theorem~\ref{tgl2gl2}, for any $\sigma \in \Aut(\C)$ we have \begin{equation}\label{proof2}\sigma(C(f, h)) = C({}^\sigma\!f, {}^\sigma\!h), \quad \sigma(C(h, g)) = C({}^\sigma\!h, {}^\sigma\!g).
\end{equation}

From~\eqref{proof1} and~\eqref{proof2}, we deduce that for all $\sigma \in \Aut(\C)$, we have the identity \begin{equation}\label{proofid1}\sigma\left(\frac{C(f, h)C(h, g)}{C(F,h)}\right) = \frac{C({}^\sigma\!f, {}^\sigma\!h)C({}^\sigma\!h, {}^\sigma\!g)}{C({}^\sigma\!F,{}^\sigma\!h)}. \end{equation}
By the basic property of Yoshida lifts (see Remark~\ref{yoshbasicrem}), we derive \begin{equation}\label{proofid2}\frac{C(f, h)C(h, g)}{C(F, h)}= -\frac{\pi^2\langle F, F\rangle}{\langle f, f\rangle} , \quad and \quad  \frac{C({}^\sigma\!f, {}^\sigma\!h)C({}^\sigma\!h, {}^\sigma\!g)}{C({}^\sigma\!F, {}^\sigma\!h)} = - \frac{\pi^2\langle {}^\sigma\!F, {}^\sigma\!F\rangle}{\langle {}^\sigma\!f, {}^\sigma\!f\rangle}.\end{equation}

Combining~\eqref{proofid1} and~\eqref{proofid2}, we deduce that $$\sigma\left(\frac{\pi^2\langle F, F\rangle}{\langle f, f\rangle}\right)=  \frac{\pi^2\langle {}^\sigma\!F, {}^\sigma\!F\rangle}{\langle {}^\sigma\!f, {}^\sigma\!f\rangle},$$
as required.

Next we consider the case when $k$ is even. Then $k+1$ is odd. Pick \emph{any} odd quadratic character $\epsilon$ and put $\chi_3 = \chi \epsilon$. Then $\chi_3$ is even and for some sufficiently large level there exists a classical holomorphic newform $h$ of weight $k+1$ and character $\overline{\chi_3}$. By normalizing, we may assume that $h$ has $\CM$ algebraic Fourier coefficients.

Now, we define $$C(F, h)= \frac{L_\mathfrak{f}((k-1)/2, \pi_F\times \pi_h)}{\pi^{5k-3} i G(\epsilon)^2 \langle F, F \rangle \langle h, h \rangle},$$
$$C(f, h)= \frac{L_\mathfrak{f}((k-1)/2, \pi_f\times \pi_h)}{\pi^{3k-1}i G(\epsilon)\langle f, f \rangle}, \quad C(h, g)= \frac{L_\mathfrak{f}((k-1)/2, \pi_h\times \pi_g)}{\pi^{2k} G(\epsilon) \langle h, h \rangle},$$
and proceed exactly as before.
\bibliographystyle{plain}
\bibliography{lfunction}

\end{document}